\documentclass{imamci}
\jno{dnaaxxx}

 \usepackage{amsmath}
  \usepackage{amsthm}
 \usepackage{amssymb}
 \usepackage{xcolor}
 \usepackage{subcaption}
 \usepackage{enumitem}

 \usepackage{tikz}
 \usepackage{pgfplots}
 \usetikzlibrary{patterns,arrows}
 \usetikzlibrary{shapes, arrows}
 \tikzstyle{block} = [draw, fill=white, rectangle, minimum height=0em, minimum width=0em]
 \tikzstyle{output} = [coordinate]
 \tikzstyle{input} = [coordinate]
 \pgfplotsset{compat=1.18}

 \usepackage[colorlinks=false,linkbordercolor=red,citebordercolor=blue,hidelinks]{hyperref}

 \newcommand{\setdef}[2]{\left\{\, #1\, \left|\, \vphantom{#1} #2 \right.\right\}}

 \newcommand{\ddt}{\tfrac{\text{\normalfont d}}{\text{\normalfont d}t}}
 
 \newcommand{\ds}{\text{\normalfont d}s}
\newcommand{\R}{\mathbb{R}}
\newcommand{\C}{\mathbb{C}}
\newcommand{\N}{\mathbb{N}}
\newcommand{\cB}{\mathcal{B}}
\newcommand{\cC}{\mathcal{C}}
\newcommand{\cD}{\mathcal{D}}
\newcommand{\cF}{\mathcal{F}}

\newcommand{\cL}{\mathcal{L}}

\newcommand{\cW}{\mathcal{W}}

\newcommand{\vp}{\varphi}
\newcommand{\ve}{\varepsilon}
\newcommand{\rp}{\mathbb{R}_{\geq 0}}

{\left(\begin{smallmatrix}}%            begin code
{\end{smallmatrix}\right)}%             end code

{\left[\begin{smallmatrix}}%            begin code
{\end{smallmatrix}\right]}%             end code

\DeclareMathOperator{\Gl}{\mathbf{Gl}}

\DeclareMathOperator{\esssup}{\rm ess\,sup}

  \usepackage{cleveref}
  \usepackage{amsthm}
 \newtheorem{Theorem}{Theorem}[section]
 \newtheorem{Lemma}{Lemma}
 
 \newtheorem{Corollary}[Theorem]{Corollary}

 \newtheorem{Remark}[Theorem]{Remark}

 \newtheorem{Assumption}{Assumption}
\newtheorem{defn}{Definition}

\crefname{Assumption}{assumption}{assumptions}
\crefname{defn}{definition}{definitions}

\usepackage[sort,square,numbers]{natbib}

\begin{document}
\title{Funnel control of linear systems with arbitrary relative degree under output measurement losses}
\author{ {\sc Thomas Berger}\\[2pt]
Institut f\"ur Mathematik, Universit\"at Paderborn, Warburger Str.~100, 33098~Paderborn, Germany, {\tt thomas.berger@math.upb.de}\\[6pt]
{\sc Lukas Lanza}\\[2pt]
Optimization-based Control Group,
Institute of Mathematics, Technische Universit\"at Ilmenau,
                          Weimarer Stra{\ss}e~25, 98693 Ilmenau,
            Germany, {\tt lukas.lanza@tu-ilmenau.de}}
\pagestyle{headings}

\maketitle

\begin{abstract}
{{
We consider tracking control of linear minimum phase systems with known arbitrary relative degree which are subject to possible output measurement losses.
We provide a control law which guarantees the evolution of the tracking error within a (shifted) prescribed performance funnel whenever the output signal is available.
The result requires a maximal duration of measurement losses and a minimal time of measurement availability, which both strongly depend on the internal dynamics of the system, and are derived explicitly. The controller is illustrated by a simulation of a mass-on-car system.
}}

{linear systems, funnel control, output tracking, measurement losses, minimum phase}
\end{abstract}

%===============================================================================

\section{Introduction}
We study output reference tracking for linear minimum phase systems with arbitrary relative degree under possible output measurement losses. Such phenomena are of significant practical relevance whenever signals are transmitted over large distances or via digital communication networks and may hence be prone to signal losses or package dropouts. In the presence of output measurement losses the performance of closed-loop control strategies may seriously deteriorate and even lead to instability. In the present paper we present a reliable strategy for linear systems which is able to guarantee a prescribed margin for the tracking error and after a period of output measurement losses it is able to recapture the error within this time-varying margin by appropriately shifting it.

Output measurement losses are typically considered within the framework of networked control systems, see e.g.~\cite{GarcBarr07,WangYang10,ClooHete10,NesiTeel04}. Within this approach, event-triggered controllers have been designed in order to guarantee global asymptotic stability, see~\cite{LehmLunz12,BlinAllg14,LinsDima19} for linear systems and~\cite{WangLemm11,DolkHeem17} for nonlinear systems. $H_\infty$ control approaches have been considered in~\cite{GaoChen08,TangGao15} and model predictive control in~\cite{MunoChri08,LjesQuev18}.
Considering systems where the output consists of sampled data is related to output measurement losses, since between two samples no additional information is available.
Recently, in~\cite{lanza2023sampleddata} a controller was developed for continuous-time systems with sampled-data output, which 
 achieves guaranteed performance, if the sampling period is sufficiently small.
Note that in this situation, the system data is available at a priori known time instances.
However, in the case of unexpected measurement losses, as far as the authors are aware, tracking control with prescribed performance bounds for the tracking error has not yet been considered. To achieve this, in the present paper we use the methodology of funnel control.

The concept of funnel control goes back to the seminal work~\cite{IlchRyan02b}, see also the survey in~\cite{BergIlch21}.  The funnel controller proved to be the appropriate tool for tracking problems in various applications such as control of industrial servo-systems~\cite{Hack17} and underactuated multibody systems~\cite{BergDrue21,BergOtto19}, control of electrical circuits~\cite{BergReis14a,SenfPaug14}, control of peak inspiratory pressure~\cite{PompWeye15}, adaptive cruise control~\cite{BergRaue18,BergRaue20} and even the control of infinite-dimensional systems such as a boundary controlled heat equation~\cite{ReisSeli15b}, a moving water tank~\cite{BergPuch22} and defibrillation processes of the human heart~\cite{BergBrei21}.

The novel funnel control design that we present in this paper relies on an intrinsic ``availability function'' which encodes (as a binary value) whether the output measurement is available at some time instant, or if the measurement is lost. 
As a consequence, no precise \textit{a priori} information about the time instants where the measurement is lost or recaptured is necessary. Then the basic idea for the control design is to employ a classical funnel controller on each interval where the output is available, set the input to zero when it is not available and restart the controller when the output signal is received again. Because we restrict ourselves to linear systems no blow-up may occur when the input is zero. The crucial obstacle in the feasibility proof of the control design in our main result Theorem~\ref{Thm:FunnnelControl} is to show that the resulting control input in the closed-loop system is globally bounded. To this end, we require appropriate assumptions on the maximal duration of measurement losses and the minimal time of measurement availability, which we explicitly derive in \Cref{Sec:ControllerDesign}. The bound for these durations essentially depends on the internal dynamics of the system~-- if the internal dynamics are absent, no restrictions must be made. However, if they are present a key step is to find an invariant set for the internal dynamics and to choose the initial width of the performance funnel large enough~-- this is elaborated in \Cref{Sec:DesignParameters}. The control design is illustrated by a simulation of a mass-on-car system in \Cref{Sec:Sim}.

\paragraph{Nomenclature.}
Throughout the present article we use the following notation, where $I \subseteq \R$ denotes an interval and $\rp := [0, \infty)$.
%$\N$ is the set of positive integers;
$\C_- :=\setdef{ z \in \C}{ {\rm Re}\, z < 0}$;
$\| x \| := \sqrt{  x^\top x }$ is the Euclidean norm of  $x \in \R^n$;
$\Gl_n(\R)$ is the set of invertible matrices~$A \in \R^{n \times n}$;
for $A \in \Gl_n(\R)$ we write $A > 0$ ($A<0$) if $A$ is positive (negative) definite;
$\sigma(A) \subseteq \C$ is the spectrum of a matrix~$A \in \R^{n \times n}$;
%$\cL_{\rm loc}^\infty (I ; \R^p)$ the set of locally essentially bounded functions $ f: I \to \R^p$;
$\cL^\infty (I; \R^p)$ is the Lebesgue space of measurable and essentially bounded functions $ f: I \to \R^p$ with norm
$\| f \|_{\infty} := \esssup_{t \in I} \|f(t)\|$;
$\cW^{k,\infty}(I ; \R^p)$ is the Sobolev space of $k$-times weakly differentiable functions $f : I \to \R^p$ such that $f,\dot f,\ldots,f^{(k)} \in \cL^\infty(I ; \R^p)$;
$\cC^k( I ; \R^p) $ is the set of $k$-times continuously differentiable functions $f : I \to \R^p$, $\cC(I ; \R^p) = \cC^0(I ; \R^p)$;
$f|_{J}$ is the restriction of $f : I \to \R^n$ to $J \subseteq I$.

\section{Problem formulation and system class}
In this section we introduce the problem under consideration, and specify the system class to which our solution applies.
Before we provide the technical details of the controller in \Cref{Sec:FeedbackLaw}, we provide a brief description of the control objective.
The overall task is output reference tracking with predefined error performance in the case that the system output is subject to measurement dropouts.
Predefined error performance means that for a system 
\begin{equation} \label{eq:System-lin}
\begin{aligned}
\dot x(t) &= A x(t) + B u(t), \quad x(0) = x^0 \in \R^n, \\
y(t) & = C x(t),
\end{aligned}
\end{equation}
the output $y(t)$ follows a given reference signal $y_{\rm ref} \in \cW^{r,\infty}(\rp,\R^m)$ with the prescribed performance
\begin{equation} \label{eq:error_performance}
    \forall \, t \ge 0 \, : \ \| y(t) - y_{\rm ref}(t) \| < \psi(t) ,
\end{equation}
for a given (time-varying) boundary function~$\psi$. 
Measurement dropouts mean, that for some intervals of length at most~$\Delta > 0$ no output measurement is available, i.e., the signal $y|_{[t,\hat t]}$ is unknown, where $|\hat t - t| \le \Delta$.
The aim is to develop a controller, which achieves~\eqref{eq:error_performance} in those intervals, where measurements are available.
Moreover, if the signal is lost, the controller is able to ensure satisfaction of~\eqref{eq:error_performance} directly after reappearance of the measurement. 
At this time instance it may be required to widen the function~$\psi$ in order to recapture the tracking error within the performance funnel, since its evolution is unknown when no measurement is available.
For a rigorous problem statement, we first introduce the system class under consideration.
\subsection{System class}
We consider linear systems~\eqref{eq:System-lin},
where $y(t) \in \R^m$ is the output, and $u(t) \in \R^m$ is the input of the system at time~$t \ge 0$.
The dynamics of system~\eqref{eq:System-lin} are governed by matrices $A \in \R^{n \times n}$ and $B, C^\top\in \R^{n\times m}$.
Note that the dimension $m \in \N$ of output and input coincide.
We assume that the system has a well-defined strict relative degree.
\begin{Assumption} \label{Ass:rel_deg}
System~\eqref{eq:System-lin} has strict relative degree $r\in\N$, i.e., $CA^kB = 0$ for all $k=0,\ldots,r-2$, and  $\Gamma := CA^{r-1} B \in \Gl_m(\R)$.
    
\end{Assumption}
Invoking \Cref{Ass:rel_deg}, the result \cite[Lem.~3.5]{ilchmann2007tracking} yields that there exist $R_i \in \R^{m \times m}$, $i=1,\ldots,r$, 
$S,P^\top \in \R^{m \times (n-rm)}$,  $Q \in \R^{(n-rm) \times (n-rm)}$ and an invertible matrix $U \in \R^{n \times n}$ such that after the coordinate transformation $(y,\ldots,y^{(r-1)},\eta) = U x$ the dynamics of system~\eqref{eq:System-lin} can equivalently be written in the form
\begin{equation} \label{eq:System}
\begin{aligned}
y^{(r)}(t) & = \sum_{i=1}^{r} R_i y^{(i-1)}(t) + S \eta(t) + \Gamma u(t), \\
\dot \eta(t) & = Q \eta(t) + P y(t),
\end{aligned}
\end{equation}
with initial conditions
\begin{equation}\label{eq:IC}
\begin{aligned}
(y(0),\ldots,y^{(r-1)}(0) ) = (y_0^0,\ldots,y_{r-1}^0 ) \in \R^{rm}, 
\quad \eta(0) = \eta^0 \in \R^{n-rm}.
\end{aligned}
\end{equation}
The second equation in~\eqref{eq:System} describes the internal dynamics of~\eqref{eq:System-lin}.
The following assumption concerns the stability of the internal dynamics of system~\eqref{eq:System-lin}, or rather system~\eqref{eq:System}, which are present if $rm < n$.
\begin{Assumption} \label{Ass:Q}
      For given numbers $M, \mu \ge 0$, the matrix $Q$ in~\eqref{eq:System} is Hurwitz, i.e., $\sigma(Q) \subset \C_-$, and satisfies 
    \begin{equation} \label{eq:Q_estimate}
        \forall \, t \ge 0 :\ \| e^{Q t} \| \le M e^{-\mu t} .
    \end{equation}
\end{Assumption}
Asking the matrix~$Q$ to be Hurwitz means assuming the internal dynamics to be stable, i.e., the system is minimum phase.
Note that if $Q$ is a given Hurwitz matrix, \eqref{eq:Q_estimate} is satisfied with $M := \sqrt{\|K^{-1}\|\|K\|}$ and $\mu := 1/(2\|K\|)$, where $K$ is a solution of the Lyapunov equation $KQ + Q^\top K = - I_{n-m}$, cf.~\cite{Godu98}.
For systems with trivial internal dynamics we set $M := 0$ and $\mu := 1$. %, which satisfies~\eqref{eq:Q_estimate}.
The next assumption is also related to the internal dynamics. The numbers $s,p \ge 0$ quantify the influence of the internal dynamics on the system dynamics.

\begin{Assumption} \label{Ass:SP}
    For given numbers~$s,p \ge 0$, the matrices $S,P$ in~\eqref{eq:System} satisfy
    \begin{align*}
        \|S\| & \le s, \quad \|P\| \le p.
    \end{align*}
\end{Assumption}
The following assumption concerns the (external) dynamics of system~\eqref{eq:System}, i.e., the matrices ${R_i \in \R^{m \times m}}$ in~\eqref{eq:System}, $i=1,\ldots,r$.
\begin{Assumption} \label{Ass:R}
For a given number~$\beta \ge 0$, and numbers~$s,p,M,\mu \ge 0$ from \Cref{Ass:Q,Ass:SP} the matrices $R_i$ in~\eqref{eq:System} satisfy
    \begin{equation}
          \sum_{i=1}^r \|R_i\|  \le \beta -  \frac{s p M - \mu}{\mu} .
    \end{equation}
\end{Assumption}
Due to the parameterization via the constants in \Crefrange{Ass:Q}{Ass:R}, we may define the following class of systems.
\begin{defn} \label{Def:SystemClass}
    For $m,r \in \N$
    a system~\eqref{eq:System} is said to belong to the class $\Sigma_{m,r}$, if \Crefrange{Ass:rel_deg}{Ass:R} are satisfied, and the symmetric part of $\Gamma = CA^{r-1}B$ is sign definite\footnote{That is, for any $v\in\R^m$ we have $v^\top (\Gamma + \Gamma^\top) v = 0$ if, and only if, $v = 0$.}; w.l.o.g. we assume $\Gamma + \Gamma^\top> 0$.
    In virtue of the equivalence of~\eqref{eq:System-lin} and~\eqref{eq:System}, we write $(A,B,C) \in \Sigma_{m,r}$.
\end{defn}

We like to note that, actually, the constants $s,p,M,\mu$ and $\beta$ in \Crefrange{Ass:Q}{Ass:R} parameterise the system class $\Sigma_{m,r}$, and hence the latter depends on the choice of these constants. 
For better readability we do not indicate this dependence explicitly.
However, it is important to note that \Crefrange{Ass:Q}{Ass:R} do not restrict the system class more than assuming it to have well defined strict relative degree and being minimum-phase.
Further note that we may also allow for $\Gamma^\top + \Gamma < 0$ by simply changing the sign in the feedback law~\eqref{def:control-scheme} defined below.

\begin{Remark}
The class of systems with strict relative degree one and trivial internal dynamics,
\begin{equation*}
    \dot x(t) = A x(t) + B u(t), \quad x(0) = x^0 \in \R^n
\end{equation*}
with $B \in \Gl_n(\R)$, and output $y(t) = x(t)$,
is solely parameterised by the number~$\beta$
 since $M=s=p=0$ and $\mu = 1$.   
\end{Remark}
\subsection{Control objective} \label{Sec:ControlObjective}
We aim to find a control scheme which achieves tracking of a given reference trajectory with prescribed transient behavior of the error, where the measurement output is subject to dropouts.
To be more precise, for a system~\eqref{eq:System} with $(A,B,C) \in \Sigma_{m,r}$, and a given reference signal $y_{\rm ref} \in \cW^{r,\infty}(\rp; \R^m)$ the output~$y$ tracks the reference in the sense that, whenever the measurement of~$y$ is available to the controller, the error~$e := y - y_{\rm ref}$ evolves within a prescribed \textit{performance funnel}
\begin{equation*}
 \cF_\vp := \setdef{ (t,e) \in \rp \times \R^m}{ \vp(t)\|e\| < 1 } ,
 \end{equation*}
where~$\vp$ determines the funnel boundary $\psi := 1/\varphi$, and belongs to the following set of monotonically increasing functions
\begin{equation*}
\Phi := \setdef{ \phi \in \cC^1(\rp; \R)}{\!
\begin{array}{l}
\forall\, t_2 \ge t_1 \ge 0: \
0 < \phi(t_1) \le\phi(t_2), \\
\exists \, d> 0 \ \forall\, t\ge 0: \
| \dot \phi(t) | \le d(1+\phi(t))
\end{array} \!}.
\end{equation*}
The performance funnel $\cF_\vp$ joins the two objectives of $e(t)$ approaching zero with prescribed transient behaviour and asymptotic accuracy. Its boundary is given by the reciprocal of $\vp$, see also \Cref{Fig:Schematic-funnel}. 

\begin{Remark}
We stress that $\vp$ may be unbounded, and in this case (and if no measurement losses occur for $t\ge T$ for some $T>0$) asymptotic tracking may be achieved, i.e., $\lim_{t\to\infty} e(t) = 0$.    
\end{Remark}

\section{Controller design} \label{Sec:ControllerDesign}
In this section we propose a novel control scheme, which achieves the control objective formulated in \Cref{Sec:ControlObjective} for any member of the system class~$\Sigma_{m,r}$.
We consider situations where the output measurement signal may be lost for some time,
and propose assumptions relating the maximal duration of measurement losses and minimal time of measurement availability.
The package dropouts in the system and the accompanying lost information of the measurements~$y(t)$ are not assumed to happen in \textit{a priori} known time intervals. We only assume that it is possible to determine, at every time instant~$t$, whether the measurement of~$y(t)$ is available or not; if the availability is not certain, then it should be rendered ``unavailable'' (this also encompasses the situation that, after a dropout, the availability of the measurement is only determined with some delay). Based on this we define an ``availability function''
\begin{equation}\label{eq:sensor}
a(t) = \begin{cases}
1, & \text{measurement of $y(t)$ available}, \\
0, & \text{measurement of $y(t)$ not available}.
\end{cases}
\end{equation}

\subsection{Availability and loss of measurement} \label{Sec:AvailableAndLoss}
In order to introduce the assumptions on the maximal duration of measurement losses and the minimal time of measurement availability we define the sequences $(t_k^-)$, $(t_k^+)$ with $t_k^\pm \nearrow\infty$ and $t_k^-< t_k^+ < t_{k+1}^- < t_{k+1}^+$ such that
\begin{equation} \label{eq:intervals}
\begin{aligned}
        \setdef{t\ge 0}{a(t) = 1} = \bigcup_{k\in\N} (t_k^+, t_{k+1}^-], \quad
    \setdef{t\ge 0}{a(t) = 0} = \bigcup_{k\in\N} (t_k^-, t_k^+],
\end{aligned}
\end{equation}
this is, on the interval $(t_k^+,t_{k+1}^-]$ the signal is available, and on the interval $(t_k^-,t_{k}^+]$ the signal is not available.
Note that it is also possible that both sequences contain only finitely many points, then either $a(t) =1$ for $t\ge t_N^+$ or $a(t)=0$ for $t\ge t_N^-$ for some $N\in\N$.

Furthermore, we require the following constants.
Choose~$q \in (0,1)$, and for $k\ge 0$ define the function
\begin{subequations} \label{eq:Ak_Ar}
\begin{equation} \label{eq:Ak}
A_k(s) = \sum_{j=0}^k s^j.
\end{equation}
Then, with $\alpha(s) := 1/(1-s)$ fix the number 
\begin{equation} \label{eq:Ar}
A_r := A_r(\alpha(q^2)) > 0.   
\end{equation}
\end{subequations}
Now, we introduce the assumptions on the maximal duration of measurement losses.
\begin{Assumption} \label{ass:signal-lost}
Choose parameters $M,\mu,s,p,\beta$ for \Crefrange{Ass:Q}{Ass:R} and consider a system $(A,B,C) \in \Sigma_{m,r}$, satisfying these assumptions.
Let $q \in (0,1)$, and $A_r$ be given by~\eqref{eq:Ar}.
The measurement signal is lost for at most~$\Delta >0$, i.e., for $t_k^\pm$ as in~\eqref{eq:intervals} we have $|t_k^- - t_k^+| \le \Delta$ for all $k \in \N$, such that 
$\Delta$ satisfies
\begin{align}
         {spM  \Delta^2} e^{\beta \Delta} &\le 1 \label{eq:Delta_1} \tag{$\Delta_1$} \\
        {spM^2  \Delta^2} e^{\beta \Delta} &<  {\frac{q}{A_r}} \label{eq:Delta_2} \tag{$\Delta_2$} \\
       {2\mu M \Delta} & < 1 \label{eq:Delta_3} \tag{$\Delta_3$}.
\end{align}
\end{Assumption}
The next assumption concerns the minimal time of measurement availability.
\begin{Assumption} \label{ass:signal-available}
{Choose parameters $M,\mu,s,p,\beta$ for \Crefrange{Ass:Q}{Ass:R} and consider a system $(A,B,C) \in \Sigma_{m,r}$, satisfying these assumptions.}
Let $q \in (0,1)$, {$A_r$ as in~\eqref{eq:Ar} and~$\Delta$ as in \Cref{ass:signal-lost}}.
The measurement signal is available for at least~$\delta>0$, i.e., for $t_k^\pm$ as in~\eqref{eq:intervals} we have $|t_{k}^+ - t_{k+1}^-| \ge \delta$ for all $k \in \N$, such that~$\delta$ satisfies
\begin{align}
        e^{\mu \delta} &\ge {\frac{4 M^2 + pM \Delta}{1-\mu M \Delta}},  \label{eq:delta_1} \tag{$\delta_1$}  \\
        e^{\mu \delta} & \ge {\frac{2spM^3 A_r \Delta e^{\beta\Delta}}{\mu q - \mu sp M^2 A_r \Delta^2 e^{\beta \Delta}}},   \label{eq:delta_2} \tag{$\delta_2$}
\end{align}
which can be satisfied because of~\eqref{eq:Delta_2} and~\eqref{eq:Delta_3}.
\end{Assumption}

\begin{Remark}\label{Rem:int-dyn}
For systems with trivial internal dynamics (the second equation in~\eqref{eq:System} is not present), \Cref{ass:signal-lost,ass:signal-available} are much weaker.
In this case we have $p=0$, $s=0$ and $M=0$ with which the inequalities~\eqref{eq:Delta_1},~\eqref{eq:Delta_2},~\eqref{eq:Delta_3} and~\eqref{eq:delta_1},~\eqref{eq:delta_2} are always satisfied, 
and hence arbitrary $\Delta>0$ and $\delta>0$ are possible so that $|t_k^- - t_k^+|\le \Delta$, and $|t_k^+ - t_{k+1}^-|\ge \delta$ for all $k\in\N$. So the only (implicit) requirement is that the sequence $(|t_k^- - t_k^+|)$ is bounded.
\end{Remark}

\subsection{{Choice of funnel boundary}} \label{Sec:DesignParameters}
In order to formulate the control law, which achieves the control objective formulated in \Cref{Sec:ControlObjective}, we introduce a funnel boundary function $\varphi_0\in\Phi$, which is defined by the following five consecutive steps.
One step is already done in~\eqref{eq:Ak_Ar}, but for the sake of completeness we restate it here.
%%%%%%%%%%%%%%%%%%%%%%%%%%%%%
In the flowchart \Cref{Fig:Parameters} the five steps towards the choice of $\vp_0 \in \Phi$ are depicted.

\begin{enumerate}[label=\emph{Step~\arabic{enumi}}., leftmargin=*]
    \item Choose $q \in (0,1)$, and set $A_r := A_r(\alpha(q^2))$ according to~\eqref{eq:Ak_Ar}.
    \item For the constants $M,\mu,p,s,\beta,\Delta,\delta$ from \Crefrange{Ass:Q}{ass:signal-available}, and
$x_{\rm ref}(\cdot) := (y_{\rm ref}(\cdot),\dot y_{\rm ref}(\cdot), \ldots, y_{\rm ref}^{(r-1)}(\cdot))$, choose $\eta^*>0$ with
\begin{subequations} \label{eq:eta-star}
\begin{align}
\eta^* & \ge {p \Delta e^{\mu\delta} \| y_{\rm ref}\|_\infty},  \label{eq:eta_star_ref} \\
    \eta^* & \ge \left(\|x_{\rm ref}\|_\infty + 1  \right) e^{\beta \Delta + \mu \delta} , \label{eq:eta_star_refall} \\
        \eta^* & \ge {\frac{pMA_r\big(\|x_{\rm ref}\|_\infty \left(1+ e^{\beta \Delta} \right) + e^{\beta \Delta}\big)}{\mu q - spM^2A_r\Delta e^{\beta\Delta}(\mu\Delta + 2Me^{-\mu\delta})}} \label{eq:eta_star_vp1}
\end{align}
\end{subequations}
where~\eqref{eq:eta_star_vp1} can be satisfied because of~\eqref{eq:delta_2}.
\item 
Let~$\vp_0 \in \Phi$ such that 
for $ E := \|x_{\rm ref}\|_\infty \left(1+ e^{\beta \Delta} \right) + e^{\beta \Delta} + sM \Delta e^{\beta \Delta} \left( 2 M e^{-\mu \delta} + {\mu\Delta} \right) \eta^*$ we have
\begin{align}
\vp_{0,\rm min} := \frac{ p M}{\mu \eta^*} & \le \vp_0(0) \le \frac{q}{A_r E} =: \vp_{0,\rm max},
 \label{eq:vp1} \tag{$\phi_1$}
\end{align}
which is possible by~\eqref{eq:eta_star_vp1}.
\item 
To exploit~\cite[Cor.~1.10]{BergIlch21}, we require the following constants.
Let $ \hat \alpha^\dagger(z) := z/(1+z)$ , and observe that $\hat \alpha^\dagger ( s \alpha(s)) = s$.
Let $\tilde \alpha(s) := 2s \alpha'(s) + \alpha(s) = (1+s)/(1-s)^2$.
Set $\mu_0 := \frac{d(1+\vp_0(0))}{\vp_0(0)}$ where $d>0$ is due to properties of~$\Phi$, and observe that $\esssup_{t \ge 0} (|\dot \vp_0(t)|/\vp_0(t)) \le \mu_0$; here we use this possibly larger constant~$\mu_0$ to guarantee that it only depends on the initial value~$\vp_0(0)$.
Then, in virtue of~\cite[Eq.~(12)]{BergIlch21}, for $k=1,\ldots,r-1$ we recursively define the constants $c_0 = 0$ and
\begin{equation} \label{eq:ci}
\begin{aligned}
e_1^0 & := \vp_0(0) e(0), \\
c_1 &:= \max \{ \|e_1^0\|^2, \hat \alpha^\dagger(1+\mu_0), q^2 \}^{1/2} < 1, \\
\mu_k & := 1 + \mu_0 \big( 1+c_{k-1} \alpha(c_{k-1}^2) \big)  + \tilde \alpha(c_{k-1}^2) \big( \mu_{k-1} + c_{k-1} \alpha(c_{k-1}^2) \big), \\
e_k^0 & := \vp_0(0) e^{(k-1)}(0) + \alpha(\|e_{k-1}^0\|^2) e_{k-1}^0, \\
c_k & := \max \{ \|e_k^0\|^2, \hat \alpha^\dagger(\mu_k) ,q^2 \}^{1/2} < 1,
\end{aligned}
\end{equation}
where $e^{(i)}(0) = y_i^0 - y_{\rm ref}^{(i)}(0)$ for $i=0,\ldots,r-1$.
Then we set
\begin{equation} \label{eq:C}
\chi := \sum_{i=1}^{r-1} c_{i} + c_{i-1} \alpha(c_{i-1}^2) + (1+c_{r-1} \alpha(c_{r-1}^2)).
\end{equation}
\item 
We refine the funnel function~$\vp_0 \in \Phi$ satisfying~\eqref{eq:vp1} such that for an intermediate $\rho \in (0,\delta)$
\begin{equation}
\vp_0(\rho) \ge \chi \label{eq:vp2} \tag{$\phi_2$} .
\end{equation}
\end{enumerate}

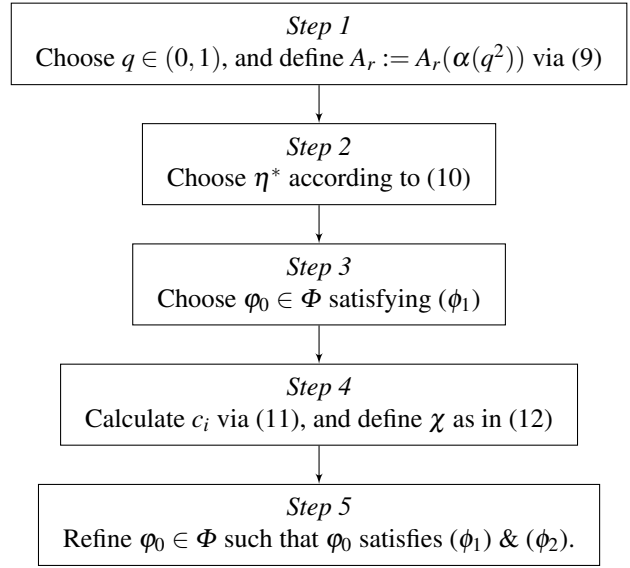
\begin{figure}[h]
\centering
\begin{tikzpicture}[auto, node distance=2cm,>=latex']
\def\hoch{0.8cm};
\def\breit{0cm};
\def\dist{1.6cm};
\node [block, minimum width = \breit, minimum height = \hoch] (q) { \begin{tabular}{c}
\emph{Step~1} \\
Choose~$q \in (0,1)$, and define $A_r:=A_r(\alpha(q^2))$ via~\eqref{eq:Ak_Ar}
\end{tabular} } ;
\node [block, minimum width = \breit, minimum height = \hoch,below of=q, node distance=\dist] (eta) { \begin{tabular}{c}
\emph{Step~2} \\
{Choose~$\eta^*$ according to~\eqref{eq:eta-star}}
\end{tabular} };
\node [block, minimum width = \breit, minimum height = \hoch,below of=eta, node distance=\dist] (vp1) { \begin{tabular}{c}
\emph{Step~3} \\
{Choose~$\vp_0 \in \Phi$ satisfying~\eqref{eq:vp1}}
\end{tabular} };
\node [block, minimum width = \breit, minimum height = \hoch,below of=vp1, node distance=\dist] (C) { \begin{tabular}{c}
\emph{Step~4} \\
Calculate~$c_i$ via~\eqref{eq:ci}, and define~$\chi$ as in~\eqref{eq:C}
\end{tabular} } ;
\node [block, minimum width = \breit, minimum height = \hoch,below of=C, node distance=\dist] (vp2) { \begin{tabular}{c}
\emph{Step~5} \\
{Refine~$\vp_0 \in \Phi$ such that $\vp_0$ satisfies~\eqref{eq:vp1}~\&~\eqref{eq:vp2}.}
\end{tabular} };
\draw[->] (q) -- (eta);
\draw[->] (eta) -- (vp1);
\draw[->] (vp1) -- (C);
\draw[->] (C) -- (vp2);

\end{tikzpicture}
\caption{Flowchart for the choice of the controller design parameters~$\eta^* \in \R$ and~$\vp_0 \in \Phi$.}
\label{Fig:Parameters}
\end{figure}

\begin{Remark}
The purpose of the constants~$q, \eta^*$ chosen in Step~1 and Step~2 of the design procedure is to determine the initial width of the performance funnel, described by the upper bound for $\vp_0(0)$ in~\eqref{eq:vp1}. Then again, condition~\eqref{eq:vp2} ensures that its width (and hence the tracking error) is not too large before the signal possibly vanishes the next time.
Note that if no internal dynamics are present, the minimal initial width of the funnel is solely determined by the reference signal and the duration of the unavailability of output measurements.
\end{Remark}

\subsection{Feedback law} \label{Sec:FeedbackLaw}
With the assumptions and definitions provided in \Cref{Sec:AvailableAndLoss,Sec:DesignParameters} we are now in the position to introduce the feedback law, which achieves the control objective defined in \Cref{Sec:ControlObjective}.
The idea for the controller design is to choose a funnel function
$\vp_0 \in \Phi$ (as in the previous subsection)
which is reset whenever $a(t)=0$.
Then, as soon as $a(t^*)=1$ for some $t^*\ge 0$ and the measurement is available again, the funnel controller from~\cite{BergIlch21} is restarted with
$\vp(t) = \vp_0(t-t^*$)
so that $\vp(t^*)>0$ and the performance funnel is sufficiently large at $t^*$ to ensure applicability of~\cite[Thm.~1.9]{BergIlch21}.
For feasibility we assume that the availability function $a(\cdot)$ from~\eqref{eq:sensor} is left-continuous and has only finitely many jumps in each compact interval.
With this, and recalling $\alpha(s) = 1/(1-s)$, we introduce the following control law for systems~\eqref{eq:System} under possible output measurement losses:
\begin{equation}\label{def:control-scheme}
\boxed{
\begin{aligned}
    \tau(t) &= \begin{cases} t, & a(t) = 0,\\ \tau(t-), & a(t) = 1,\end{cases} \\
    \vp(t) &= \begin{cases} 0, & a(t) = 0,\\ \vp_{0}(t-\tau(t)), & a(t) = 1,\end{cases}\\
    e_1(t) &= \vp(t) e(t) = \vp(t)\big(y(t) - y_{\rm ref}(t)\big),\\
    e_{i+1}(t) &= \vp(t) e^{(i)}(t) \!  + \!  \alpha( \|e_i(t)\|^2 ) e_i(t) , \ i=1,\ldots,r-1,\\
    u(t) &= - a(t) \alpha(\|e_r(t)\|^2)  e_r(t).
\end{aligned}
}
\end{equation}
%\end{small}
With $\tau(t-)$ we denote the left limit $\tau(t-) = \lim_{h\searrow 0} \tau(t-h)$  of the piecewise continuous function $\tau$ at~$t$. This ensures that $\tau$ is constant on any interval where $a(t)=1$ (i.e., the measurements are available), and so the necessary time shift of~$\varphi_0$ does not increase further. Note that if $\Gamma^\top + \Gamma < 0$ (instead of $\Gamma^\top + \Gamma > 0$ as in \Cref{Def:SystemClass}), then we may simply change the sign in the control and obtain $u(t) = a(t) \alpha(\|e_r(t)\|^2)  e_r(t)$.

If the output measurement is always available, i.e., $a(t)=1$ for all $t\ge 0$, then the controller~\eqref{def:control-scheme} coincides with that proposed in~\cite{BergIlch21} and the existence of a global solution of the closed-loop system follows from the results presented there.
Since it is not known \textit{a priori} when output measurement losses occur, the funnel function $\vp$ cannot be globally defined in advance.
Therefore,~$\vp$ is defined online as part of the control law~\eqref{def:control-scheme}; it is equal to a shifted version of the reference funnel function~$\vp_{0}$
whenever measurements are available, and zero otherwise.
Note that the loss of the system's output signal possibly introduces a discontinuity in the control signal. 
A typical choice for a funnel function is $\vp_0(t) = (a e^{-bt} + c)^{-1}$ with $a,b,c > 0$, which is depicted in \Cref{Fig:Schematic-funnel}.
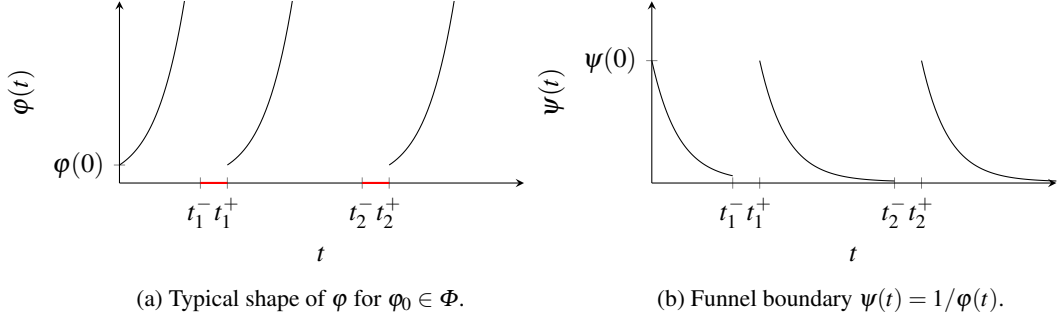
\begin{figure}
\centering
\begin{subfigure}[b]{0.48\linewidth}
\begin{tikzpicture}
% \tikzmath{
% % Zeitpunkte und Funktionsparameter
% \t1 = 3;
% \t2 = 4;
% \t3 = 9;
% \t4 = 10;
% \t5 = 15;
% \A = 1;
% \B = 1;
% \D = 0.01;
% \VP = 1/(\A+\D);
% }
\begin{axis}[axis lines = left, xlabel = \(t\), ylabel near ticks, ylabel = {\(\varphi(t)\)},
xmin=0, xmax=15,
ymin=0, ymax=10,
xtick={3,4,9,10},
xticklabels={$t_1^-$,$t_1^+$,$t_2^-$,$t_2^+$},
ytick={0.99},
yticklabel={$\vp(0)$},
height = 4cm,
width = \columnwidth,
]
\addplot[domain=0:3,samples=50,color=black] {1/(1*exp(-1*x) + 0.01)} ;
\addplot[domain=3:4,samples=50,color=red, line width=0.3mm] {0} ;
\addplot[domain=4:9,samples=50,color=black] {1/(1*exp(-1*(x-4)) + 0.01)} ;
\addplot[domain=9:10,samples=50,color=red, line width=0.3mm] {0} ;
\addplot[domain=10:15,samples=50,color=black] {1/(1*exp(-1*(x-10)) + 0.01)} ;
\end{axis}
\end{tikzpicture}
\caption{Typical shape of~$\vp$ for $\vp_0 \in \Phi$.}
\label{Fig:Overall-funnel}
\end{subfigure}
\begin{subfigure}[b]{0.48\linewidth}
\begin{tikzpicture}
%\tikzmath{
% Zeitpunkte
% \t1 = 3;
% \t2 = 4;
% \t3 = 9;
% \t4 = 10;
% \t5 = 15;
% \A = 1;
% \B = 1;
% \D = 0.1;
% \PSI = \A+\D;
% }
\begin{axis}[axis lines = left, xlabel = \(t\), ylabel near ticks, ylabel = {\( \psi(t)\)},
xmin=0, xmax=15,
ymin=0, ymax=1.5,
xtick={3,4,9,10},
xticklabels={$t_1^-$,$t_1^+$,$t_2^-$,$t_2^+$},
ytick = 1.01,
yticklabel = {$\psi(0)$},
height = 4cm,
width = \columnwidth,
]
\addplot[domain=0:3,samples=50,color=black] {1*exp(-1*x) + 0.01} ;
\addplot[domain=4:9,samples=50,color=black] {1*exp(-1*(x-4)) + 0.01)} ;
\addplot[domain=10:15,samples=50,color=black] {1*exp(-1*(x-10)) + 0.01} ;
\end{axis}
\end{tikzpicture}
\caption{Funnel boundary $\psi(t) = 1/\vp(t)$.}
\label{Fig:Overall-funnel-boundary}
\end{subfigure}
\caption{Schematic shape of a typical funnel boundary with shifts.}
\label{Fig:Schematic-funnel}
\end{figure}
\section{Main result}
Now we are in the position to formulate our main result.
To phrase it, the application of the controller~\eqref{def:control-scheme} to a system~\eqref{eq:System-lin} (or equivalently a system~\eqref{eq:System}) with $(A,B,C) \in \Sigma_{m,r}$
under possible output measurement losses leads to a closed-loop initial-value problem which has a global solution. 
By a solution of~\eqref{eq:System},~\eqref{def:control-scheme} on $[0,\omega)$ we mean a function $(y,\eta)\in \cC^{r-1}([0,\omega),\R^m) \times \cC([0,\omega),\R^{n-rm})$ with $\omega\in (0,\infty]$, which satisfies the initial conditions~\eqref{eq:IC} and $(y^{(r-1)}, \eta)|_{[0,\omega)}$ is locally absolutely continuous and satisfies the differential equation in~\eqref{eq:System} with~$u$ defined by~\eqref{def:control-scheme} for almost all $t\in[0,\omega)$.
The solution $(y,\eta)$ is called maximal, if it has no right extension that is also a solution.

\begin{Theorem} \label{Thm:FunnnelControl}
Choose parameters $M,\mu,s,p,\beta$ for \Crefrange{Ass:Q}{Ass:R} and consider a system~\eqref{eq:System} with $(A,B,C) \in \Sigma_{m,r}$, satisfying these assumptions.
Let $y_{\rm ref} \in \cW^{r,\infty}(\rp; \R^m)$ be a given reference, {initial values as in~\eqref{eq:IC},}
$a(\cdot)$ be an availability function as in~\eqref{eq:sensor}, which is left-continuous and has only finitely many jumps in each compact interval, and choose design parameters $\eta^*$ as in~\eqref{eq:eta-star}, and $\vp_{0}\in\Phi$ satisfying~\eqref{eq:vp1},~\eqref{eq:vp2}.
If the initial conditions
\begin{subequations} \label{eq:initials}
 \begin{align}
\forall\, i = 1,\ldots,r \, : \ \| e_i(0)\| & < 1, \label{eq:initial_ei} \\
 \|\eta^0\| &\le \eta^* \label{eq:initial_eta}
 \end{align}
 \end{subequations}
are satisfied,
then the control scheme~\eqref{def:control-scheme} applied to system~\eqref{eq:System} yields an initial-value problem which has a solution,
every solution can be extended to a maximal solution and every maximal solution $(y,\eta): [0,\omega) \to \R^m\times\R^{n-rm}$ has the following properties:
\begin{enumerate}[label = (\roman{enumi}), ref=(\roman{enumi}),leftmargin=*]
\item \label{ome_inf}
the solution is global, i.e., $\omega = \infty$,
\item \label{error_funnel}
 the tracking error~$e(t) = y(t) - y_{\rm ref}(t)$ evolves within the funnel boundaries, i.e., $\varphi(t) \|e(t)\| < 1$ for all $t\ge 0$,
\item \label{bounded_u}
the control signal is globally bounded, i.e., $u \in \cL^\infty(\rp; \R^m)$; moreover, $y \in \cW^{r,\infty}(\rp; \R^m)$.
\end{enumerate}
\end{Theorem}
The proof is relegated to the appendix.
The proof is constructive and we provide an explicit global bound for the control input~$u$.

\begin{Remark}
The maximal duration of measurement losses~$\Delta$, the minimal time of measurement availability~$\delta$ and the lower bound for~$\eta^*$ in~\eqref{eq:eta-star} depend on the system parameters.
We emphasize that $\eta^*$, which is a bound for the initial internal state, may be chosen larger than in~\eqref{eq:eta-star}. This results in a larger initial width of the funnel boundary due to~\eqref{eq:vp1}.
\end{Remark}

In view of Remark~\ref{Rem:int-dyn} we present the following result as a direct consequence of Theorem~\ref{Thm:FunnnelControl}.

\begin{Corollary}
Consider a system~\eqref{eq:System} with $(A,B,C) \in \Sigma_{r,m}$ with trivial internal dynamics, i.e., we have $n=rm$, and the second equation in~\eqref{eq:System} is absent, and with initial conditions $(y_0^0,\ldots,y_{r-1}^0) \in \R^{rm}$.
Let~$y_{\rm ref} \in \cW^{r,\infty}(\rp; \R^m)$, and $a(\cdot)$ be an availability function as in~\eqref{eq:sensor} which is left-continuous and has only finitely many jumps in each compact interval.
Let $ \Delta> 0$ be an arbitrary long duration of possible signal losses, and $\delta> 0$ be an arbitrary short duration of guaranteed signal availability.
Choose the design parameter~$\eta^*$ as in~\eqref{eq:eta-star}, where $s=p=M=0$ and~$\mu =1$ in \Cref{ass:signal-lost,ass:signal-available}.
Further, choose $\vp_{0}\in\Phi$ satisfying~\eqref{eq:vp1}, \eqref{eq:vp2}.
If the initial conditions~\eqref{eq:initial_ei} are satisfied,
then the control scheme~\eqref{def:control-scheme} applied to system~\eqref{eq:System} yields an initial value problem which has a solution,
every solution can be extended to a maximal solution and every maximal solution~$y: [0,\omega) \to \R^m $ has the properties~\ref{ome_inf}--\ref{bounded_u} from \Cref{Thm:FunnnelControl}.
\end{Corollary}

\section{Simulation}\label{Sec:Sim}
To illustrate the action of the proposed controller, we numerically simulate an application of the controller~\eqref{def:control-scheme} to a system~\eqref{eq:System}.
We consider the \textit{mass-on-car} system introduced in~\cite{SeifBlaj13}, where on a car with mass~$m_1$ (in kg) a ramp is mounted on which a mass~$m_2$ (in kg), coupled to the car by a spring-damper-component with spring constant~$k > 0$ (in N/m) and damping~$d>0$ (in Ns/m), passively moves; a control force~$F = u$ (in N) can be applied to the car.
The situation is depicted in \Cref{Fig:Mass-on-a-car}.
\begin{figure}
\begin{center}
\includegraphics[width=0.6\textwidth]{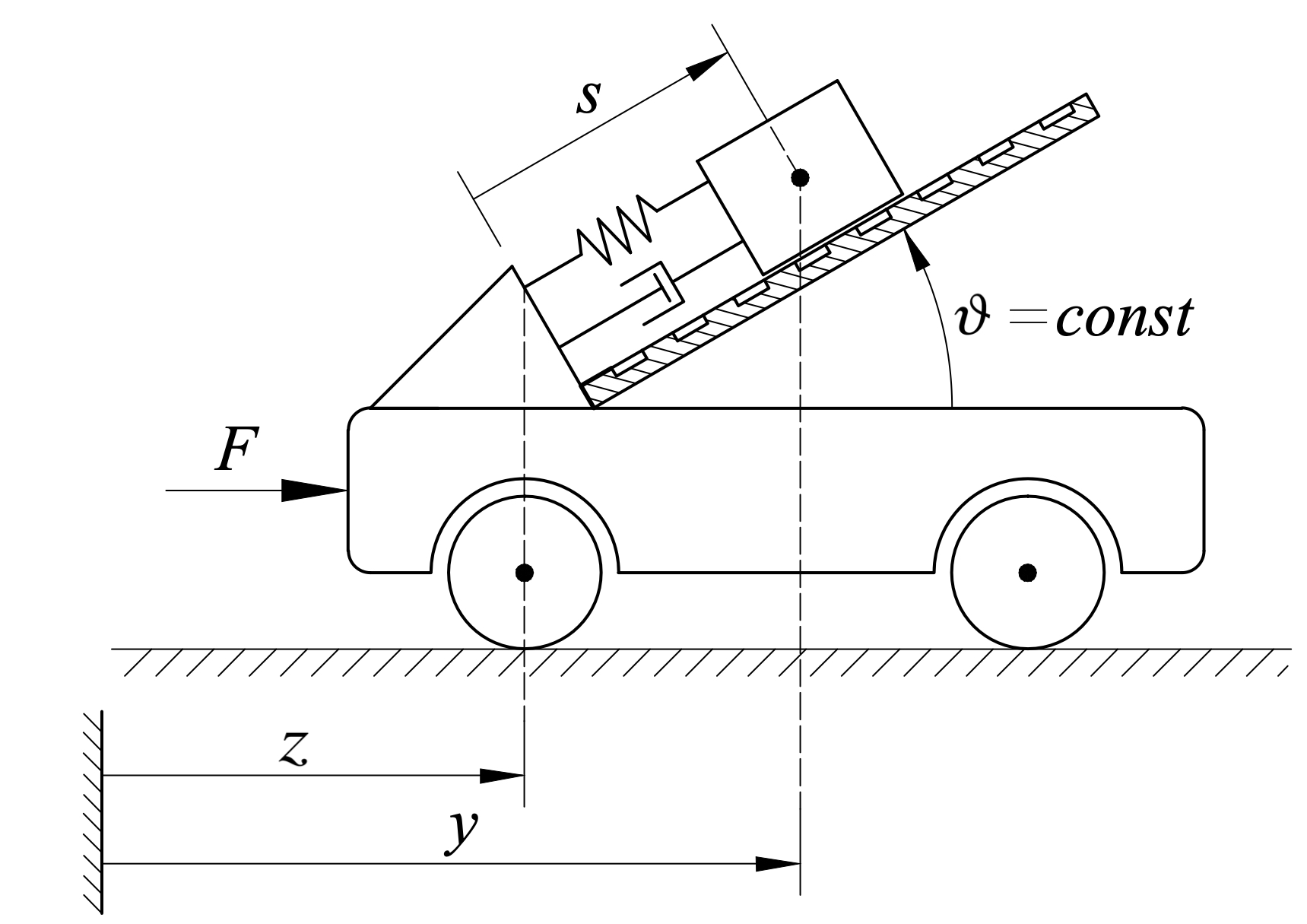}
\end{center}
\caption{Mass-on-car system. The figure is taken from~\cite{BergIlch21}.}
\label{Fig:Mass-on-a-car}
\end{figure}
The equations of motion for the system read
\begin{subequations} \label{eq:Mass-on-car-system}
\begin{equation}  \label{eq:Mass-on-car}
\begin{bmatrix}
m_1 + m_2 & m_2 \cos(\vartheta) \\ m_2 \cos(\vartheta) & m_2
\end{bmatrix}\!
\begin{pmatrix}
\ddot z(t) \\
\ddot s(t)
\end{pmatrix} \!+\!
\begin{pmatrix}
0 \\
ks(t) \!+\! d \dot s(t)
\end{pmatrix} \!=\!
\begin{pmatrix}
u(t) \\
0
\end{pmatrix}\!,
\end{equation}
with the horizontal position of the second mass~$m_2$ as output
\begin{equation} \label{eq:Mass-on-car-output}
y(t) = z(t) + \cos(\vartheta) s(t).
\end{equation}
\end{subequations}
For the simulation we choose the parameters $m_1 = 4$, $m_2 = 1$, $k=2$, $d = 1$, $\vartheta = \pi/4$ and the initial values $z(0)=s(0)=\dot z(0) = \dot s(0) = 0$.
As a reference signal we choose~$y_{\rm ref}: \rp \to \R$, $t \mapsto \cos(t)$, by which $\|y_{\rm ref}\|_\infty = \|x_{\rm ref}\|_\infty = 1$.
As elaborated in~\cite[Sec.~3]{BergIlch21}, for the above parameters system~\eqref{eq:Mass-on-car-system} has relative degree two with respect to the output~\eqref{eq:Mass-on-car-output}, and hence belongs to~$\Sigma_{1,2}$.
Thus, it can equivalently be written in the form~\eqref{eq:System} with $r=2$, and
\begin{equation*}
 R_1 = 0, \ R_2 = \frac{8}{9}, \
S = \frac{-4\sqrt{2}}{9} \begin{bmatrix} 2 & 1 \end{bmatrix},\ \Gamma = \frac19, \
 Q = \begin{bmatrix} 0 & 1 \\ -4 & -2  \end{bmatrix}, \
P = 2 \sqrt{2} \begin{bmatrix} 1 \\ 0 \end{bmatrix}.
\end{equation*}
\Cref{Ass:Q} is satisfied with $\mu = 0.3305$ and $M = 2.2477$.
According to \Cref{ass:signal-lost,ass:signal-available} with $q=0.95$, 
we assume
$\Delta \le 5.01 \cdot 10^{-2} \, \rm s$ and
$\delta \ge 18.8 \, \rm s$.
Condition~\eqref{eq:eta-star} is satisfied with $\eta^* = 133 \, 145$.
We choose $\vp_0(t) = (ae^{-bt} + c)^{-1}$.
According to~\eqref{eq:vp1} the funnel function has to satisfy
\begin{equation*}
\vp_{0,\rm min} =  1.4449 \cdot 10^{-4} \le \vp_0(0) \le  1.4449 \cdot 10^{-4} = \vp_{0,\rm max} ,
\end{equation*}
and we choose $c = 0.03$, $a = 1/\vp_{0, \rm min} - c$, and $b = 1$.
Then, the constant from \eqref{eq:C} is given as $\chi = 21.4683$, and condition~\eqref{eq:vp2} is satisfied with $\vp(\rho) = 33$, where $\rho = 0.99 \, \delta$.\\
We simulate output tracking over the interval $0-60$ seconds.
The simulation has been performed in \textsc{Matlab} (solver: \textsf{ode23tb}).
For illustration purposes we consider two losses and reappearances of the output signal.
\begin{figure}[ht]
\centering
         \includegraphics[width=0.46\linewidth]{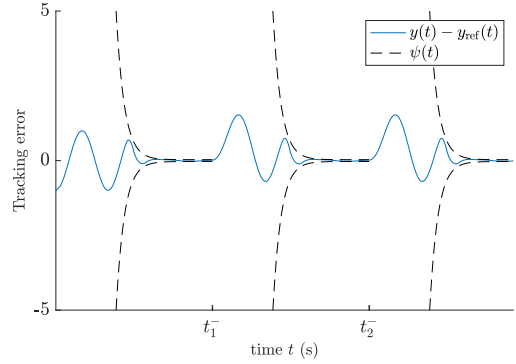}
         \caption{Error between the output~$y$ and the reference signal~$y_{\rm ref}$, and funnel boundary~$\psi = 1/\vp$.}
         \label{Fig:Error}
\end{figure}
\Cref{Fig:Error} shows the error~$e = y - y_{\rm ref}$ between the system output and the reference signal.
As expected, the error evolves within the prescribed funnel boundaries whenever the output signal is available, and remains bounded whenever the signal is not available.
In \Cref{Fig:Control} the control input is depicted. 
It can be seen that on large time intervals, especially after $t_1^-$ and $t_2^-$, the input signal is zero. Only when the performance funnel gets tighter again a large control action is necessary, which induces some small peaks in the input when a small tracking error is enforced. But even in the presence of measurement losses the control input is bounded and the evolution of the tracking error within the (shifted) performance funnel is guaranteed.
\begin{figure}[ht]
\centering
\begin{minipage}[t]{0.45\textwidth}
         \includegraphics[width=\linewidth]{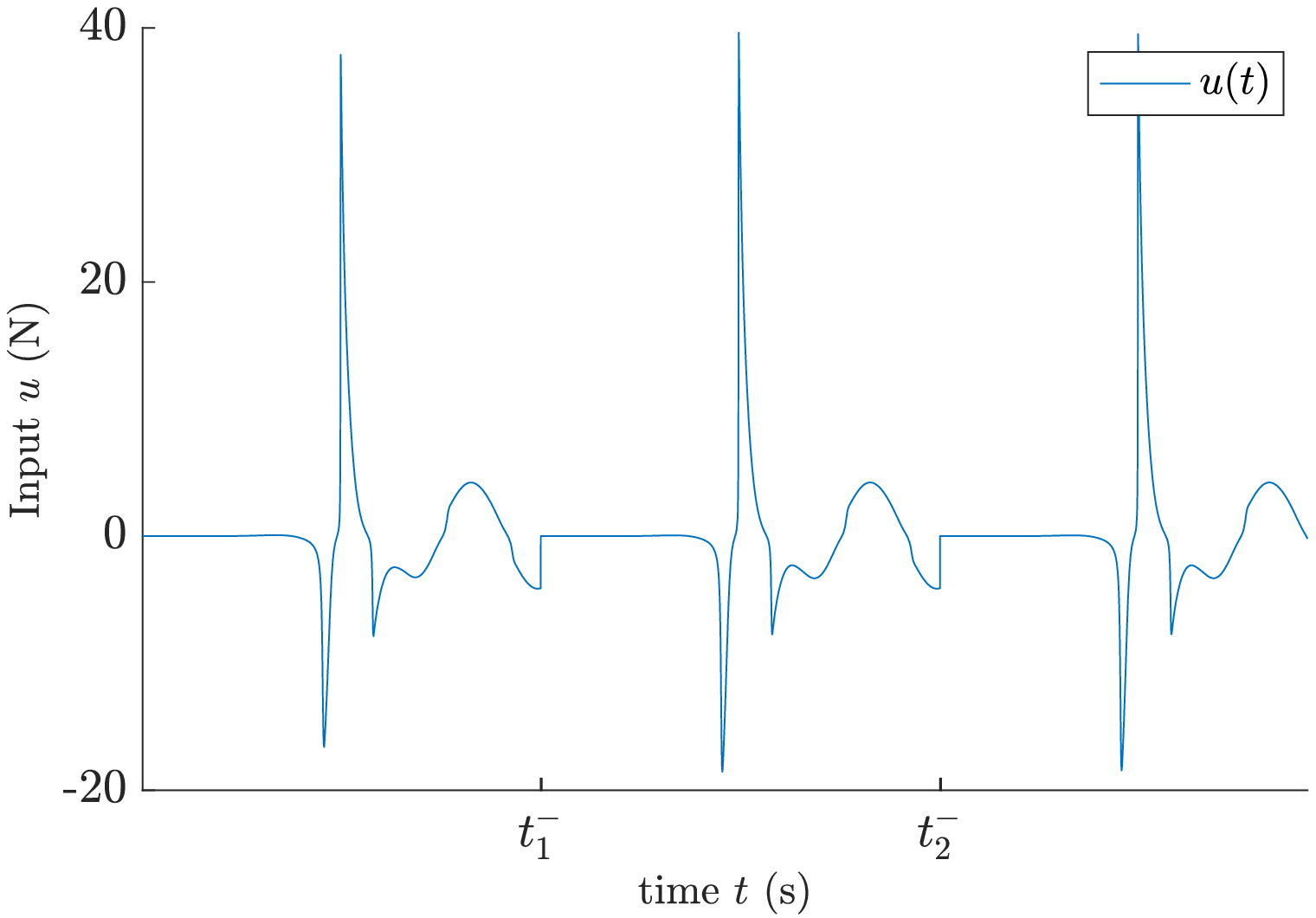}
         \caption{Control input~$u$.}
         \label{Fig:Control}
\end{minipage}
\hfill
\begin{minipage}[t]{0.45\textwidth}
         \includegraphics[width=\linewidth]{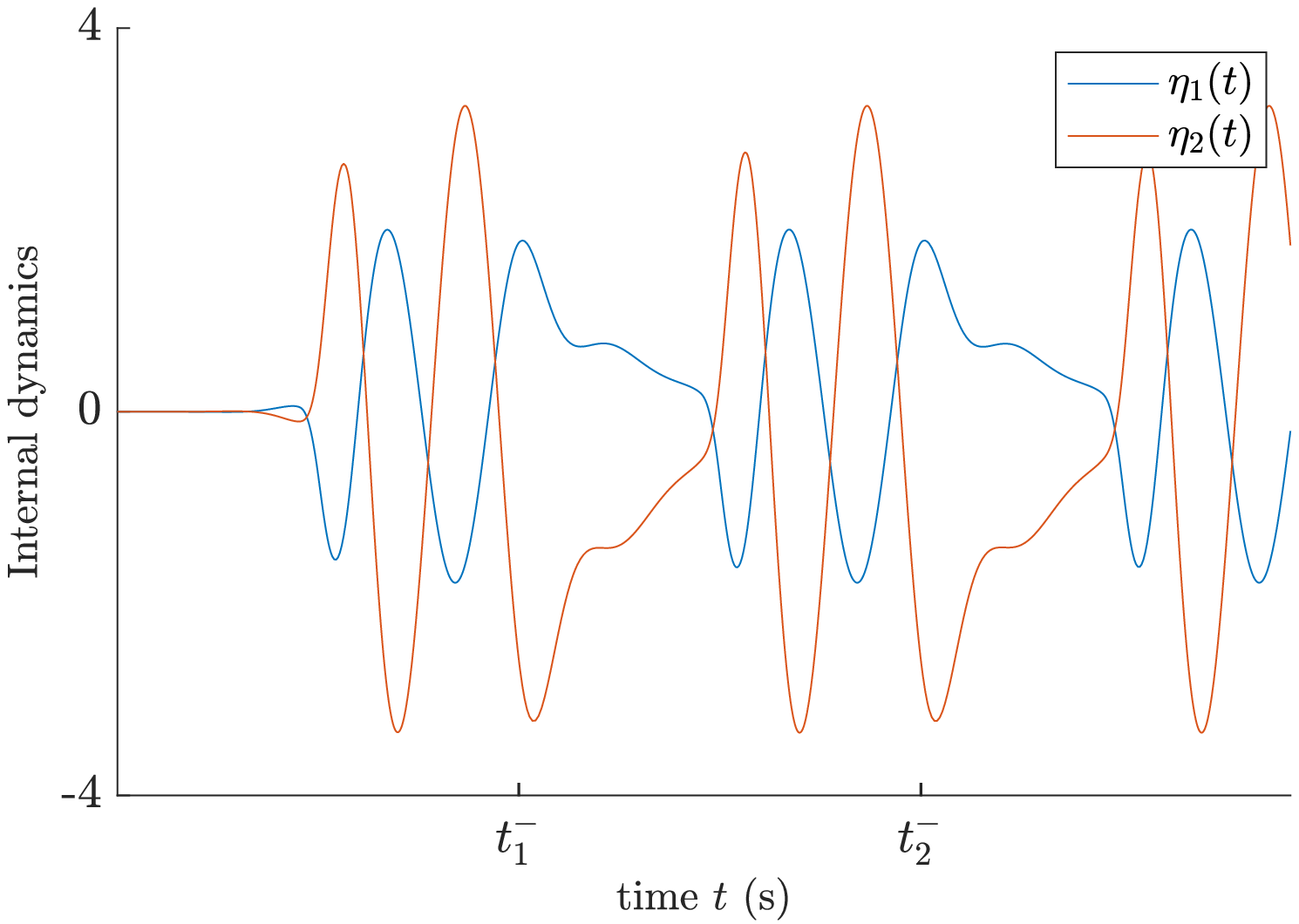}
         \caption{Evolution of the internal dynamics.}
         \label{Fig:ID}
\end{minipage}
\end{figure}
In \Cref{Fig:ID} the evolution of the state of the internal dynamics is depicted (available in simulation, but not used in the controller).
It can be seen that the internal dynamics strongly influence the evolution of the system output during periods of measurement losses.
This illustrates the importance of incorporating the internal dynamics in the theoretical estimates.
Note that, however, the state $\eta(t)$ is much smaller than the theoretical bound $\eta^*$, which is quite conservative.

\ \\
\textbf{\Cref{ass:signal-lost,ass:signal-available}, and conditions~\eqref{eq:vp1},~\eqref{eq:vp2} are conservative.}
To demonstrate that the controller also works in situations where the estimates in \Cref{ass:signal-lost,ass:signal-available} on the measurement availability are not satisfied, we run a second simulation, where we choose a much larger $\Delta$ (measurement lost) and much smaller $\delta$ (measurement guaranteed). Moreover, we choose a funnel boundary, which is much tighter at the measurement reappearance than prescribed by~\eqref{eq:vp1}.
The results are depicted in \Cref{Fig:Error_lesscons,Fig:Control_lesscons}, where we used the following parameters:
the signal is lost for $\Delta = 2$, and guaranteed available only for $\delta = 3$; the funnel function is given by
$\varphi_0(t) = \big(a e^{-bt} + c\big)^{-1}$, where $a=5$, $b=1$, and $c=0.2$.
\begin{figure}[ht]
\centering
\begin{minipage}[t]{0.45\textwidth}
\includegraphics[width=\linewidth]{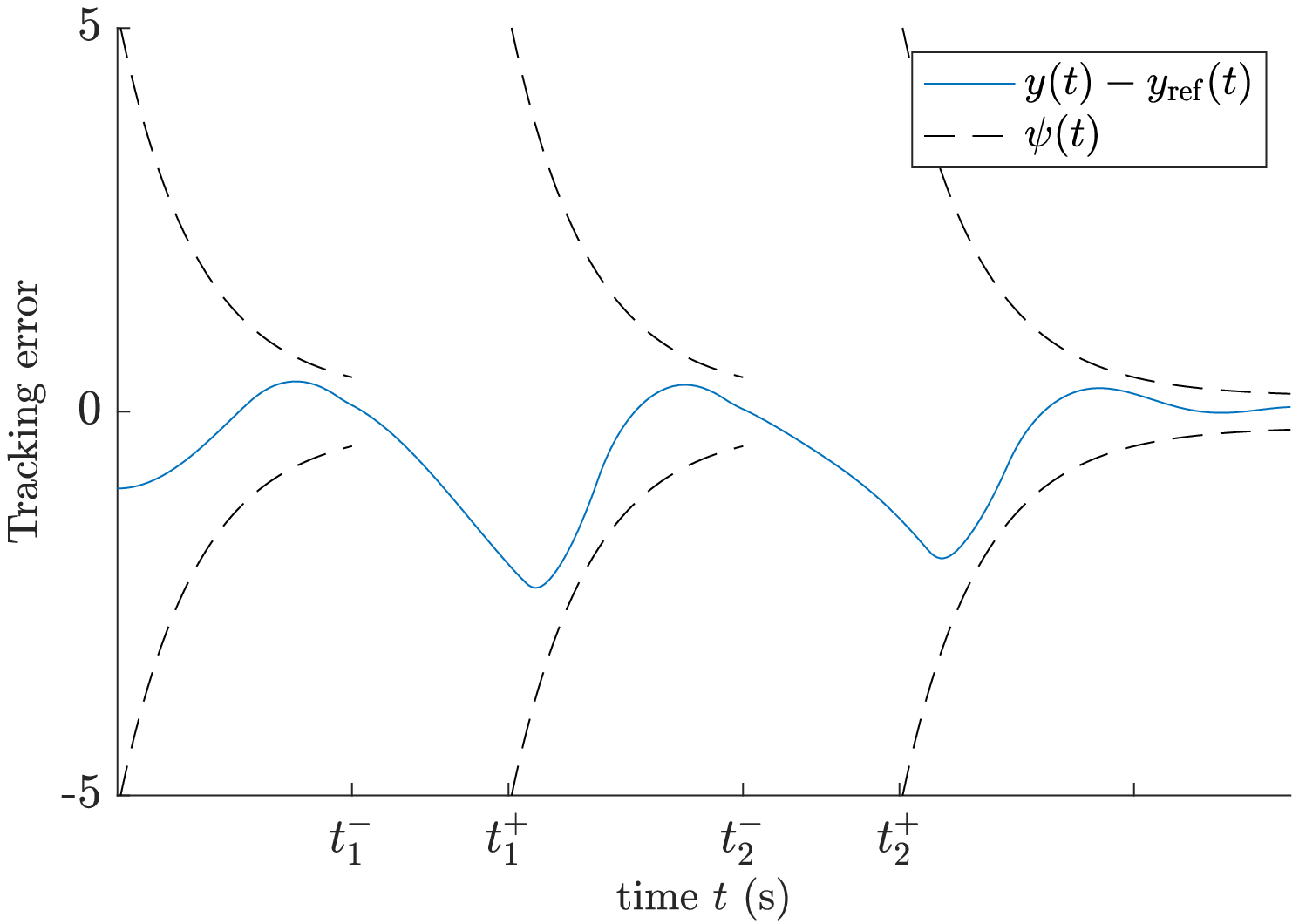}
         \caption{Evolution of the tracking error for less conservative estimates, and funnel boundary~$\psi = 1/\vp$.}
         \label{Fig:Error_lesscons}
\end{minipage}
\hfill
\begin{minipage}[t]{0.45\textwidth}
         \includegraphics[width=\linewidth]{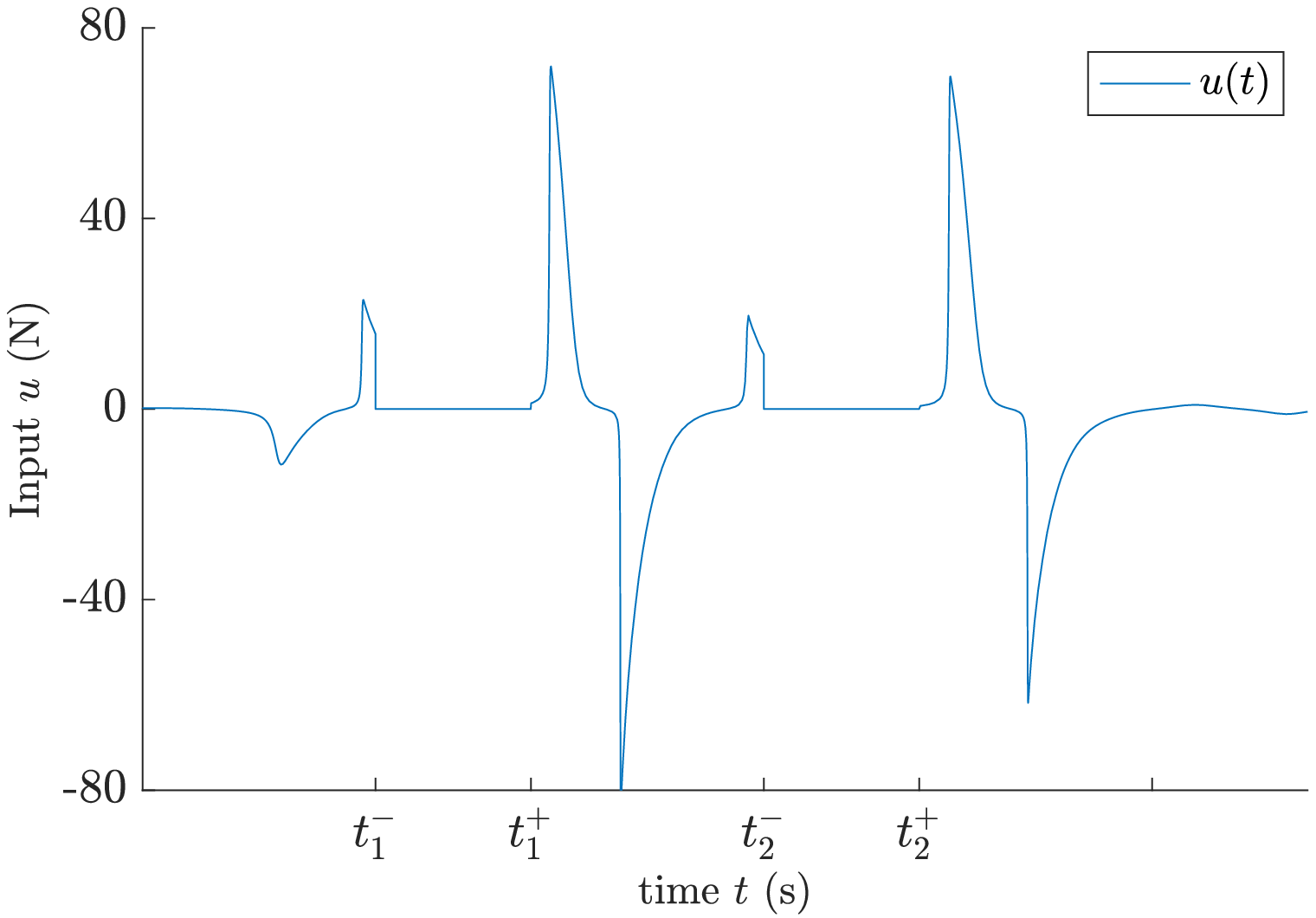}
         \caption{Control input~$u$ for less conservative estimate.}
         \label{Fig:Control_lesscons}
\end{minipage}
\end{figure}
Although \Cref{ass:signal-lost,ass:signal-available} and conditions~\eqref{eq:vp1},~\eqref{eq:vp2} are not satisfied, respectively, the proposed controller~\eqref{def:control-scheme} achieves the control objective~\eqref{eq:error_performance} even for long measurement losses.
This very clearly illustrates that the estimates required for the proof of Theorem~\ref{Thm:FunnnelControl} are worst case estimates.

\section{Conclusion}

In the present paper we introduced a novel funnel controller for output reference tracking of linear minimum phase systems which are prone to losses of the output measurements.
We proved that the closed-loop system has a global solution, and the presented feedback law achieves a prescribed transient behavior of the tracking error within a (shifted) performance funnel and all involved signals are bounded. 
In particular, the input signal is bounded, and the maximal control value can be computed in advance invoking the parameters which define the system class. 
Feasibility of the control requires a maximal duration of measurement losses~$\Delta$ and a minimal time of measurement availability~$\delta$, for both of which upper and lower bounds, respectively, have been derived explicitly. However, these bounds are conservative (as can be seen by the numerical example in \Cref{Sec:Sim}) and further research is necessary to find better estimates.\\
Another topic for future research is the extension of the results to nonlinear systems. Regarding this, it is clear that some kind of Lipschitz condition is required for the system, because otherwise a blow-up of the solutions cannot be excluded on time-intervals where the output measurement is not available.
Furthermore, the controller performance might be improved by including available knowledge of system parameters, e.g., applying a suitable non-zero open-loop control signal on intervals where no output measurement is available.

\section*{Acknowledgements}
Funded by the Deutsche Forschungsgemeinschaft (DFG, German Research
Foundation) – Project-IDs 362536361 and 471539468.

\appendix
\section{Technical lemmas}
We provide some technical results to be used in the proof of Theorem~\ref{Thm:FunnnelControl}.
First, we record that, if \Cref{Ass:Q} is satisfied, then we have for $t\ge t_0 \ge 0$
\begin{align}
\int_{t_0}^{t} \| e^{Q (s-t_0)}\| \ds & \le \frac{M}{\mu} ( 1 - e^{-\mu (t-t_0)}) \le \frac{M}{\mu}, \label{eq:exp_const}\\
\int_{t_0}^{t} \| e^{Q (s-t_0)}\| \ds & \le M \int_{t_0}^{t}  e^{-\mu(s-t_0)}  \ds \le M (t-t_0) . \label{eq:exp_lin}
\end{align}
Further, we recall that the second of equations~\eqref{eq:System} has the solution
\begin{equation} \label{eq:sol_eta}
\eta(t) = e^{Q(t-t_0)} \eta(t_0) + \int_{t_0}^{t} e^{Q(t-s)} P y(s) \ds
\end{equation}
and, hence, for any signal $y \in \cL^\infty(\rp;\R^m)$ we have
\begin{equation} \label{eq:eta_fist_estimatation}
\| \eta(t) \| \le M e^{-\mu (t-t_{0})} \|\eta(t_0)\| +  \|P\|\, \| y|_{[t_0,t]} \|_\infty  \int_{t_0}^{t}  \| e^{Q (s-t_0)}\| \ds .
\end{equation}
We derive a lemma which provides an exponential bound for the solution of~\eqref{eq:System} whenever no measurement is available.

\begin{Lemma} \label{lem:x_lost_bounded}
Choose parameters $M,\mu,s,p,\beta$ for \Crefrange{Ass:Q}{Ass:R} and consider a system~\eqref{eq:System} with $(A,B,C) \in \Sigma_{m,r}$, satisfying these assumptions. 
Then for all solutions $(y,\eta)\in \cC^{r-1}([0,\omega),\R^m) \times \cC([0,\omega),\R^{n-rm})$, $\omega\in (0,\infty]$, of System~\eqref{eq:System}
with $u|_{(t_0,t_1)} = 0$ for $0\le t_0 < t_1 \le \omega$ and with $x = (y^\top,\dot y^\top,\ldots,(y^{(r-1)})^\top)^\top$ we have that for all $t\in[t_0,t_1)$
\begin{equation*}
\begin{aligned}
\|x|_{[t_0,t]}\|_\infty \le \left( \|x(t_0)\| + s M \|\eta(t_0)\| \int_{t_0}^{t} e^{-\mu(\tau-t_0)} {\rm d} \tau \right) e^{ \beta (t-t_0)}.
\end{aligned}
\end{equation*}
\end{Lemma}
\begin{proof}
Let $x= (x_1^\top,\ldots,x_r^\top)^\top$ and set $w(t) := \|x|_{[t_0,t]}\|_\infty$ for $t\in [t_0,\omega)$. Then we have that
\[
    \dot x(t) = \begin{pmatrix} x_2(t)\\ \vdots \\ x_r(t) \\  \sum_{i=1}^r R_i x_i(t) + S \eta(t) \end{pmatrix}
\]
for almost all $t\in[t_0,t_1]$ and upon integration we obtain
\[
    \|x(t)\| \!\le\! \|x(t_0)\| +\! \int_{t_0}^t \|x(\tau)\| + \sum_{i=1}^r \|R_i\| \|x_i( \tau)\| + s \|\eta(\tau)\| \,{\rm d}\tau.
\]
Then, using~\eqref{eq:exp_const} and~\eqref{eq:sol_eta}, we have
\begin{align*}
  w(t) %\le \|x(t_0)\| + \sup_{\tau\in[t_0,t]} \int_{t_0}^\tau \|x(\tau)\|  + \|f(\oT(x)(\tau),0)\| \, {\rm d} \tau\\
  &\le  \|x(t_0)\| + \sup_{r\in[{t_0},t]} \int_{t_0}^r \Bigg[ w(\tau) + s \|e^{Q(\tau-t_0)} \eta(t_0)\| \\
  &\quad + \sum_{i=2}^r \|R_i\| \|x_i|_{[t_0,\tau]}\|_\infty   \!+\! \left(\|R_1\| + \frac{s p M}{\mu} \right) \|x_1|_{[t_0,\tau]}\|_\infty  \Bigg] \, {\rm d} \tau \\
  &\le \|x(t_0)\| +\int_{t_0}^t \underset{\le \beta}{\underbrace{\Bigg[ 1 + \left(\sum_{i=1}^r \|R_i\|  +  \frac{s p M}{\mu} \right) \Bigg] }} w(\tau)\,{\rm d} \tau
  +  s M \|\eta(t_0)\| \int_{t_0}^{t} e^{-\mu(\tau-t_0)} \,{\rm d} \tau.
\end{align*}
The assertion then follows from Gr\"onwall's lemma.
%\qed
\end{proof}
The second lemma provides a technical estimate used in the proof of the main result.
\begin{Lemma} \label{Lem:estimation_Ar}
For $k =0,\ldots,r$, $r \in \N$, let~$A_k$ be given by~\eqref{eq:Ak}, and $q \in (0,1)$.
Let $\ell : [0,1) \to [1,\infty)$ be a bijection, and $\lambda, E \ge 0$ with
\begin{equation} \label{eq:est_vp0}
E \lambda \le \frac{q}{A_r(\ell(q^2))}.
\end{equation}
Further let $\xi_0,\ldots,\xi_{r-1} \in \R^n$ with
\begin{equation} \label{eq:est_ei}
\forall\, k \in \{0,\ldots,r-1\}: \ \| \xi_{k} \| \le E,
\end{equation}
and define $\zeta_0:= 0$ and $\zeta_{k+1} \in \R^n$ for $k=0,\ldots,r-1$ by
\begin{equation} \label{def:zeta_i}
\zeta_{k+1} := \lambda \xi_{k} + \ell(\|\zeta_k\|^2) \zeta_k .
\end{equation}
Then
\begin{equation*}
\forall\, k \in \{1,\ldots,r\} : \ \|\zeta_k\| \le \lambda E A_{k-1}(\ell(q^2)) \le q.
\end{equation*}
\end{Lemma}
\begin{proof}
First observe that for $s \ge 0$ we have
\begin{equation*} 
\forall\, k \in \N  :\ A_k(s) \le A_k(s) + s^{k+1} = A_{k+1}(s).
\end{equation*}
Furthermore, for $\tilde A_k : = A_k \big(\ell(q^2) \big)$ we have that
\[
   \lambda E \tilde A_k \le \lambda E A_r(\ell(q^2)) \stackrel{\eqref{eq:est_vp0}}{\le} q.
\]
Finally, we show that
\begin{equation} \label{eq:est_ei_Ai}
\forall\, k \in \{1,\ldots,r\} :\ \|\zeta_k\| \le  \lambda E \tilde A_{k-1}
\end{equation}
by induction over $k$. For $k=1$ we have
\begin{equation*}
\begin{aligned}
\|\zeta_1\| & \overset{\eqref{def:zeta_i}}{\le } \lambda \|\xi_0\| \overset{\eqref{eq:est_ei}}{\le } \lambda E .
\end{aligned}
\end{equation*}
Let~\eqref{eq:est_ei_Ai} be true for some $k \in \{1,\ldots,r-1\}$.
Then, using monotonicity of~$\ell(\cdot)$, we obtain
\begin{equation*}
\begin{aligned}
\|\zeta_{k+1}\| & \overset{\eqref{def:zeta_i}}{\le } \lambda \| \xi_{k} \| + \ell(\|\zeta_k \|^2) \|\zeta_k\|  \overset{\eqref{eq:est_ei},\eqref{eq:est_ei_Ai}} {\le } \lambda E + \ell \big((\lambda E \tilde A_{k-1})^2\big)   \lambda E \tilde A_{k-1} \\
& \le \lambda E \big( 1 + \ell( q^2) \tilde A_{k-1} \big)  = \lambda E \big( 1 + \ell( q^2) A_{k-1} (\ell(q^2) ) \big) = \lambda E A_k(\ell(q^2)),
\end{aligned}
\end{equation*}
where we have used that $1+sA_{k-1}(s) = A_k(s)$. This proves~\eqref{eq:est_ei_Ai}.
\end{proof}

\section{Proof of Theorem~\ref{Thm:FunnnelControl}} \label{PROOF} 
\begin{proof}
The proof consists of four consecutive steps. \\
\emph{Step 1.} First, we establish the existence of a solution of~\eqref{eq:System},~\eqref{eq:IC},~\eqref{def:control-scheme}.
With $x_{\rm ref}$ as defined in \Cref{Sec:DesignParameters} and following {Step~1} in the proof of~\cite[Thm.~1.9]{BergIlch21}, we introduce $\cB = \setdef{w\in\R^m}{\|w\|<1}$ and for $\alpha(s)=1/(1-s)$ the map
\[
    \gamma: \cB\to\R^m,\ w\mapsto \alpha(\|w\|^2) w,
\]
and with this the sets $\cD_k$ and maps $\rho_k : \cD_k \to \cB$, $k=1,\ldots,r$ recursively as follows:
\begin{equation*} %\label{def:D_and_rho}
\begin{aligned}
\cD_1 &:= \cB, \ \rho_1 : \cD_1 \to \cB, \ \zeta_1 \mapsto \zeta_1, \\
\cD_k &:= \setdef{ (\zeta_1,\ldots,\zeta_k) \in \R^{km}}{ 
Z := (\zeta_1,\ldots,\zeta_{k-1}) \in \cD_{k-1}, \
\zeta_k + \gamma(\rho_{k-1}(Z)) \in \cB}, \\
\rho_{k} &: \cD_k \to \cB, \ (\zeta_1,\ldots,\zeta_k) \mapsto \zeta_k + \gamma(\rho_{k-1}(\zeta_1,\ldots,\zeta_{k-1})).
\end{aligned}
\end{equation*}
With this we define the set
\begin{equation*}
\cD := \setdef{(t,\xi) \in \rp \times \R^{rm}}{ \vp(t)\|\xi-x_{\rm ref}(t)\| \in \cD_r }
\end{equation*}
and $\rho:\cD\to\cB,\ (t,\xi)\mapsto \rho_r\big(\vp(t)\big(\xi-x_{\rm ref}(t))\big)$. Since $a(\cdot)$ is left-continuous the set $\cD$ is relatively open.
Then, $u$ in~\eqref{def:control-scheme} satisfies
\[
    u(t) = - a(t) \alpha(\|\rho(t,x(t))\|^2) \rho(t,x(t)).
\]
For $\xi = (\xi_1,\ldots,\xi_r)$ we formally define the function $F: \cD \times \R^{n-rm} \to \R^{n}$ by
\[
F(t,\xi,\eta) = \Big(\xi_2,\ldots, \xi_{r},
\sum_{i=1}^r R_i \xi_i + S \eta -a(t) \alpha(\|\rho(t,\xi)\|^2) \rho(t,\xi), Q\eta + P\xi_1 \Big) .
\]
Then we obtain with $x(\cdot) := (y(\cdot),\dot y(\cdot),\ldots,y^{(r-1)}(\cdot))$ an initial value problem
\begin{equation} \label{eq:equiv-ivp}
\begin{aligned}
\begin{pmatrix} \dot x(t)\\ \dot \eta(t)\end{pmatrix} &= F\left(t, x(t),\eta(t) \right), \\
x(0) &= \big(y_0^0,\ldots,y_{r-1}^0\big),\ \eta(0)=\eta^0,
\end{aligned}
\end{equation}
which is equivalent to~\eqref{eq:System},~\eqref{eq:IC},~\eqref{def:control-scheme}.
Note that~$F$ is continuous in~$(\xi_1,\ldots,\xi_{r},\eta)$ and locally essentially bounded and, in particular, measurable in the variable~$t$ regardless of the possible discontinuities of~$a(\cdot)$.
Therefore, since $(0,x(0)) \in \cD$, a straightforward adaption of~\cite[Thm.~B.1]{IlchRyan09} to the current context yields the existence of a maximal solution $(x,\eta): [0,\omega) \to \R^{n}$ of~\eqref{eq:equiv-ivp}, where $\omega \in (0,\infty]$.
Moreover, the closure of the graph of the solution of~\eqref{eq:equiv-ivp} is not a compact subset of~$\cD\times \R^{n-rm}$.

\emph{Step 2.} We establish~\ref{error_funnel} on $[0,\omega)$.
To this end, let $(t_k^-)$, $(t_k^+)$ be as in~\eqref{eq:intervals}.
It is also possible that both sequences contain only finitely many points, then either $a(t) =1$ for $t\ge t_N^+$ or $a(t)=0$ for $t\ge t_N^-$ for some $N\in\N$; the following arguments apply, \textit{mutatis mutandis}, in both cases.
We define $\textbf{e}(\cdot) := x(\cdot) - x_{\rm ref}(\cdot)$. Since we consider a subclass of the system class under consideration in~\cite{BergIlch21},
and since by~\eqref{eq:initial_ei} we have $\vp(0) \textbf{e}(0) \in \cD_r$, the result~\cite[Thm.~1.9]{BergIlch21} restricted to the interval~$[0,t_{1}^-]$ is applicable and ensures assertion~\ref{error_funnel} for $t \in [0,t_{1}^-] \subseteq [0,\omega)$, the inclusion since without measurement losses~\cite[Thm.~1.9]{BergIlch21} yields $\omega = \infty$.
Further, since by construction we have $\vp|_{[t_1^-,t_1^+)} = 0$, assertion~\ref{error_funnel} is true for $t \in [t_1^-,t_1^+) \subseteq [0,\omega)$, the inclusion via standard theory of (linear) differential equations since $u|_{[t_1^-,t_1^+)} = 0$.
In order to reapply~\cite[Thm.~1.9]{BergIlch21} at $t = t_1^+$, we establish that the initial conditions~\eqref{eq:initials} are satisfied for~$t=t_1^+$.
First, we show~\eqref{eq:initial_ei} at $t_1^+$.
We set $\psi(\cdot) := 1/\vp_0(\cdot)$, then we find that
\begin{equation} \label{eq:eta_t1-}
\begin{aligned}
\|\eta(t_1^-)\| & \overset{\eqref{eq:initial_eta},\eqref{eq:exp_lin},\eqref{eq:eta_fist_estimatation}}{ \le} M e^{- \mu \delta} \eta^* + {pM\Delta} \left( \psi(0) + \|y_{\rm ref}\|_\infty \right) 
\overset{\eqref{eq:eta_star_ref}}{\le} 2 M e^{-\mu \delta} \eta^* + {pM\Delta} \psi(0).
\end{aligned}
\end{equation}
By~\cite[Cor.~1.10]{BergIlch21} we have for all $i=0,\ldots,r-2$, and the constants defined in~\eqref{eq:ci} that
\begin{equation}\label{eq:est-e(i)}
\forall\,t \in [0,t_1^-) \, : \ \|e^{(i)}(t)\| \le \psi(t) (c_{i+1} + c_i \alpha(c_i^2)),
\end{equation}
and moreover, since $\| e_r(t) \| < 1$ for $t \in [0,t_1^-)$, we have $\| e^{(r-1)}(t) \| \le \psi(t) ( 1 +  c_{r-1} \alpha(c_{r-1}^2))$.
Hence, $\| \textbf{e}(t) \| \le \chi \psi(t) $ for $\chi$ defined in~\eqref{eq:C}, and in particular,
\begin{equation} \label{eq:Cor1_10}
\| \textbf{e}(t_1^-) \| \le \chi \psi(t_1^-) \le \chi \psi(\rho) {\stackrel{\eqref{eq:vp2}}{\le} 1}
\end{equation}
for $\rho < \delta \le t_1^-$ as in~\eqref{eq:vp2} since $\psi$ is monotonically decreasing by properties of~$\Phi$.
With this, using Lemma~\ref{lem:x_lost_bounded} we obtain
\begin{equation}\label{eq:x_t1-t1+}
\begin{aligned}
\|x|_{[t_1^-,t_1^+]}\|_\infty &\le  \Bigg( \|x(t_1^-)\|   + s M \|\eta(t_1^-)\| \int_{t_1^-}^{t_1^+} e^{-\mu (s-t_1^-)} \ds \Bigg) e^{\beta (t_1^+-t_1^-)} \\
& \overset{\eqref{eq:exp_lin}}{\le}  \Bigg( \| \textbf{e}(t_1^-) \| + \|x_{\rm ref}\|_\infty   + s M \Delta \| \eta(t_1^-) \| \Bigg) e^{\beta (t_1^+-t_1^-)} \\
& \overset{\eqref{eq:Cor1_10}}{\le} \|x_{\rm ref}\|_\infty  e^{\beta \Delta}
 + \left( {1} + s M \Delta \| \eta(t_1^-) \| \right) e^{\beta \Delta}.
\end{aligned}
\end{equation}
Therefore,
\begin{equation} \label{eq:e_t1+}
    \begin{aligned}
        \| \textbf{e}(t_1^+) \| 
        & \le \|x_{\rm ref} (t_1^+)\| \! +\! \|x(t_1^+)\|  \le \|x_{\rm ref}\|_\infty + \|x|_{[t_1^-,t_1^+]}\|_\infty \\
& \stackrel{\eqref{eq:x_t1-t1+}}{\le}  \|x_{\rm ref}\|_\infty \left(1+ e^{\beta \Delta} \right)  + \left( {1} + s M \Delta \| \eta(t_1^-) \| \right) e^{\beta \Delta}  \\
& \stackrel{\eqref{eq:eta_t1-}}{\le} \|x_{\rm ref}\|_\infty \left(1+ e^{\beta \Delta} \right) + e^{\beta \Delta} + s M \Delta \left( 2 M e^{-\mu \delta} \eta^* + {pM\Delta} \psi(0) \right) e^{\beta \Delta} \\
& \stackrel{\eqref{eq:vp1}}{\le} \|x_{\rm ref}\|_\infty \left(1+ e^{\beta \Delta} \right) + e^{\beta \Delta} + sM \Delta e^{\beta \Delta} \left( 2 M e^{-\mu \delta} + {\mu\Delta} \right) \eta^*
=  E.
    \end{aligned}
\end{equation}
Invoking~\eqref{eq:vp1}, Lemma~\ref{Lem:estimation_Ar} (applied with $\lambda = \vp(t_1^+) = \vp_0(0)$) yields
\begin{equation} \label{eq:ei_t1+}
\begin{aligned}
\|e_i(t_1^+)\| \le q \le c_i < 1, \ i=1,\ldots,r-1, \quad
\|e_r(t_1^+)\| \le q.
\end{aligned}
\end{equation}
Therefore,~$\vp(t_1^+) \textbf{e}(t_1^+) \in \cD_r$.
Furthermore, via~\eqref{eq:eta_t1-}, using Lemma~\ref{lem:x_lost_bounded}
and~\eqref{eq:Cor1_10} we obtain with similar estimates as above
\begin{equation} \label{eq:eta_t1+}
    \begin{aligned}
        \| \eta(t_1^+) \| & \overset{\eqref{eq:exp_lin},\eqref{eq:eta_fist_estimatation}}{ \le} M \| \eta(t_1^-)\| + {pM\Delta} \| y|_{[t_1^-,t_1^+]}\|_\infty   \\
& \overset{\eqref{eq:x_t1-t1+}}{\le}  M \| \eta(t_1^-)\| + {pM\Delta} \left( \|x_{\rm ref}\|_\infty  e^{\beta \Delta}  + \left( {1} + {sM\Delta} \| \eta(t_1^-) \| \right) e^{\beta \Delta}\right) \\ 
&  \overset{\eqref{eq:eta_star_refall}}{\le}  \left( M + {spM^2 \Delta^2} e^{\beta \Delta} \right)  \| \eta(t_1^-)\| + {pM\Delta} e^{-\mu \delta} \eta^*  \\
& \overset{\eqref{eq:eta_t1-}}{\le} \left( M +  {spM^2 \Delta^2} e^{\beta \Delta} \right) \left( 2 M e^{-\mu \delta} \eta^* + {pM\Delta} \psi(0) \right) + {pM\Delta} e^{-\mu \delta} \eta^*  \\
& \overset{\eqref{eq:Delta_1},\eqref{eq:vp1}}{\le}  \left( 4 M^2 e^{-\mu \delta} + {2M\mu\Delta + pM \Delta e^{-\mu \delta}} \right) \eta^* 
\overset{\eqref{eq:delta_1}}{\le}  \eta^*.
    \end{aligned}
\end{equation}
Therefore, the initial conditions~\eqref{eq:initials} are satisfied at $t = t_1^+$ and~\cite[Thm.~1.9]{BergIlch21} is applicable for~$t \ge t_1^+$.
Moreover, invoking~\eqref{eq:ei_t1+}, the estimates~\eqref{eq:eta_t1-},~\eqref{eq:e_t1+} and~\eqref{eq:eta_t1+} are valid for $t=t_2^-$ and $t=t_2^+$, respectively, since $\|\eta(t_1^+)\| \le \eta^*$ and $[t_1^+,t_2^+] \subseteq [0,\omega)$ via the same arguments as above.
Therefore, we obtain the following chain of inductive implications \\

\begin{tikzpicture}[every text node part/.style={align=center}, >=implies, scale=0.9    ]
\centering
% \tikzmath{
% \r = 2.8 ;
% \a = 7;
% }
    %
\foreach \c in {0,...,6}
{
  \node (u\c) at ({-\c*360/7+90}:2.8) (\c) {};
}
\draw[double, thick,->] (0) node[above]{$\vp(t_k^+) \textbf{e}(t_k^+) \in \cD_r$ \\ and $\|\eta(t_k^+)\|\le\eta^*$} -- node[right=0.7em,above]{\eqref{eq:initials}} (1) ;
\draw[double, thick,->] (1) node[right]{funnel control applicable \\ \quad for $t \in [t_k^+,t_{k+1}^-) \subseteq [0,\omega)$} -- node[right]{\eqref{eq:eta-star},~\eqref{eq:vp1},~\eqref{eq:vp2}} (2);
\draw[double, thick,->] (2) node[right]{$\|\eta(\! t_{k+1}^- \!)\|$\,satisfies\,\eqref{eq:eta_t1-}} -- node[below right]{\eqref{eq:e_t1+}} (3) ;
\draw[double, thick,->] (3) node[ below right, node distance = 4em]{ $\|\textbf{e}(t_{k+1}^+)\| \le E$ }  -- node[below]{\eqref{eq:ei_t1+}} (4);
\draw[double, thick,->] (4) node[below left]{$\forall\, i=1,\ldots,r:$ \\ $\|e_i(t_{k+1}^+)\| \le q  < 1$} -- node[below left]{\eqref{eq:vp1}} (5) ;
\draw[double, thick,->] (5) node[left]{$\vp(t_{k+1}^+) \, \textbf{e}(t_{k+1}^+) \in \cD_r$}  -- node[left]{\eqref{eq:eta_t1+}} (6) node[left]{$\|\eta(t_{k+1}^+)\| \le \eta^*$};
\draw[->, >=stealth] (6) -- node[above left]{$k \to k+1$} (0);
\end{tikzpicture}

Summarising, this means that funnel control can be reapplied at $t = t_k^+$ for all $k \in \N$ with $[t_k^+,t_{k+1}^-) \subseteq [0,\omega)$.
This yields~\ref{error_funnel} on $[0,\omega)$.

%%%%%%%%%%%%%%%%%%%%%%%%%%%%%%%%%%%%%%%%%%%%%%%%
%
\emph{Step 3.} We show $y \in \cW^{r,\infty}([0,\omega);\R^m)$ and $u \in \cL^{\infty}([0,\omega);\R^m)$.
Invoking~\eqref{eq:est-e(i)} and~\eqref{eq:x_t1-t1+} we obtain $y \in \cW^{r-1,\infty}([0,\omega);\R^m)$. To obtain a global bound for~$u$ and~$y^{(r)}$ let 
\[
    Y_{\rm max} := \max_{i=0,\ldots,r} \|y_{\rm ref}^{(i)}\|_\infty,\quad \lambda := \inf_{t\ge 0} \psi(t),
\]
$\gamma>0$ such that $\tfrac12 v^\top (\Gamma + \Gamma^\top) v \ge \gamma \|v\|^2$ for all $v\in\R^m$, and recall $\tilde \alpha(s) = (1+s)/(1-s)^2$. Further set
\[
    \bar \eta := \max\left\{\eta^*, M \eta^* + \frac{pM}{\mu} \big(\psi(0) + Y_{\max}\big)\right\},
\]
and observe that 
\[
    \|\eta(t)\|  \overset{\eqref{eq:initial_eta},\eqref{eq:exp_const},\eqref{eq:eta_fist_estimatation}}{ \le} M e^{- \mu (t-t_k^+)} \|\eta(t_k^+)\| + \frac{pM}{\mu} \left( \psi(0) + \|y_{\rm ref}\|_\infty \right) 
    \le  M \eta^* + \frac{pM}{\mu} \left( \psi(0) + Y_{\max} \right) \le \bar \eta
\]
for all $t\in [t_k^+,t_{k+1}^-]$, and that $\|\eta(t)\|\le \eta^* \le \bar \eta$ by a similar estimate as in~\eqref{eq:eta_t1+} for $t\in [t_k^-,t_k^+]$. Define with $c_i$ from~\eqref{eq:ci} the constant
\begin{align*}
 \tilde C:= \mu_0 \left(\!1\!+\! \frac{c_{r-1}}{1-c_{r-1}^2} \! \right) \!+\! \tilde \alpha(c_{r-1}^2) \left(\! \mu_{r-1} \!+\! \frac{c_{r-1}}{1-c_{r-1}^2} \!\right)  \!+\! \sum_{i=1}^r \|R_i\| \left(\!1\!+\! \frac{c_{i-1}}{1-c_{i-1}^2} + \frac{Y_{\rm max}}{\lambda} \! \right) +  \frac{s}{\lambda} \bar \eta  + \frac{Y_{\rm max}}{\lambda}.
\end{align*}
Let $\ve \in (0,1)$ be the unique point such that $\frac{\tilde C}{\gamma \vp_0(0)} = \frac{\ve}{1-\ve^2}$. Then, we define
\[
    c_r :=\max\left\{\|e_r^0\|^2, \ve, q^2 \right\}^{1/2} < 1.
\]
We show that $\|e_r(t)\|\le c_r$ for all $t \in [0,t_1^-)$. %(or on any interval of existence, respectively).
Suppose there exists $t_1\in [0,t_1^-)$ such that $\|e_r(t_1)\| > c_r$ and define
\[
    t_0 := \max\setdef{t\in [0,t_1]}{ \|e_r(t)\| = c_r},
\]
which is well-defined since $\|e_r(0)\| \le c_r$. 
First observe that, by the same calculations as in the proof of~\cite[Cor.~1.10]{BergIlch21}, we have
for the auxiliary expression $\gamma_{r-1}(t) := \alpha(\|e_{r-1}(t)\|^2) e_{r-1}(t)$ that
\[
    \|\dot \gamma_{r-1}(t)\| \le \tilde \alpha(c_{r-1}^2) \left(\mu_{r-1} + \alpha(c_{r-1}^2) c_{r-1} \right).
\]
Furthermore, since $\|e_r(t)\|\ge c_r $ for all $t\in[t_0,t_1]$ we have $\alpha(\|e_r(t)\|^2) \ge 1/(1-c_r^2)$. 
Hence, we may calculate
\begin{align*}
  \tfrac12 \ddt \|e_r(t)\|^2 &= e_r(t)^\top \big( \dot \vp(t) e^{(r-1)}(t) \!+\! \vp(t) e^{(r)}(t) \!+\! \dot \gamma_{r-1}(t)\big)\\
  &\le \|e_r\| \left( \mu_0 \vp(t) \|e^{(r-1)}(t)\| \!+\! \tilde \alpha(c_{r-1}^2) \left(\mu_{r-1} + \frac{c_{r-1}}{1-c_{r-1}^2}\right)   \right.\\
  &\quad \left.+ \vp(t) Y_{\rm max} + \vp(t) \left(\sum_{i=1}^r \|R_i\| \|y^{(i-1)}(t)\| + s \bar \eta\right)\right)\\
   &\quad - \tfrac12 \vp(t) \alpha(\|e_r(t)\|^2) e_r(t)^\top (\Gamma+\Gamma^\top) e_r(t)\\
   &\le \|e_r\| \left( \mu_0 \vp(t) \|e^{(r-1)}(t)\| \!+\! \tilde \alpha(c_{r-1}^2) \left(\mu_{r-1} + \frac{c_{r-1}}{1-c_{r-1}^2}\right) \right.\\
  &\quad \left.+ \frac{Y_{\rm max}}{\lambda} + \sum_{i=1}^r \|R_i\| \left(1+ \frac{c_{i-1}}{1-c_{i-1}^2} + \frac{Y_{\rm max}}{\lambda}\right) + \frac{s}{\lambda} \bar \eta\right) - \frac{\gamma \vp(0)}{1-c_r^2} \|e_r(t)\|^2 \\
   &\le \left( \tilde C -  \gamma \vp(0) \frac{ c_r}{1-c_r^2}\right) \|e_r(t)\| \le 0,
\end{align*}
by which $c_r < \|e_r(t_1)\| \le \|e_r(t_0)\| = c_r$, a contradiction.
By~\eqref{eq:ei_t1+} we have that $\|e_r(t_k^+)\|\le q \le c_r$ for all~$k \in \N$ with $t_k^+ \in [0,\omega)$. Therefore, the arguments above can be reapplied on any interval $[t_k^+,t_{k+1}^-) \subseteq [0,\omega)$ to achive $\| e_r(t)\| \le c_r$ for all $t \in [t_k^+,t_{k+1}^-)$. Then, invoking $u|_{[t_{k}^-,t_k^+)} = 0$, it follows from~\eqref{def:control-scheme} that $\|u(t)\| \le c_r/(1-c_r^2)$ for all $t\in [0,\omega)$, thus $u \in \cL^\infty([0,\omega);\R^m)$. As a consequence, it follows from~\eqref{eq:System} that $y^{(r)}\in \cL^\infty([0,\omega);\R^m)$.

\emph{Step 4.}
We show that the solution is global. Suppose the opposite, i.e., $\omega < \infty$.
Then, since $\|\eta(t)\|\le \bar \eta$ and for all $i=1,\ldots,r$ we have $\|e_i(t)\|\le c_i<1$ for $t\in [t_k^+,t_{k+1}^-)$ by~\cite[Cor.~1.10]{BergIlch21} and \emph{Step~3}, and $\|e_i(t)\|\le q \le c_i$ for $t\in [t_k^-,t_k^+)$ by~\eqref{eq:ei_t1+} (note that it is straightforward to extend the estimate~\eqref{eq:e_t1+} to $t\in [t_k^-,t_k^+)$), it follows that the closure of the graph of the solution of~\eqref{eq:equiv-ivp} is a compact subset of $\cD\times \R^{n-rm}$, which contradicts the findings of \emph{Step 1}.
This yields assertion~\ref{ome_inf} and consequently assertions~\ref{error_funnel}~\&~\ref{bounded_u} follow.
This completes the proof.
\end{proof}

\end{document}